\documentclass[11pt]{article}

\usepackage{fullpage}
\usepackage{parskip}
\usepackage{tikz} 
\usepackage{hyperref}
\usepackage{eqnarray,amsthm, amssymb, amsmath,verbatim,epsfig}
\usepackage{mathrsfs}
\usepackage[margin=1in]{geometry}
\usepackage{enumerate}
\usepackage{graphicx}
\usepackage{epstopdf}

\usetikzlibrary{matrix,trees,arrows, positioning, automata}

\usetikzlibrary{positioning}
\usetikzlibrary{fit}
\usetikzlibrary{patterns}

\newtheorem{theorem}{Theorem}

\newtheorem{corollary}[theorem]{Corollary}

\newcommand{\tref}[1]{Theorem \ref{#1}}

\newcommand{\fref}[1]{Figure \ref{fig:#1}}
\newcommand{\OP}{\mathcal{OP}}

\newcommand{\des}{\textit{des}}
\newcommand{\Des}{\mathrm{Des}}
\newcommand{\red}{\mathrm{red}}
\newcommand{\sg}{\sigma}
\newcommand{\WOP}{\mathcal{WOP}}
\newcommand{\pdes}{\textit{pdes}}
\newcommand{\prise}{\textit{prise}}
\newcommand{\mindes}{\textit{mindes}}
\newcommand{\minrise}{\textit{minrise}}
\newcommand{\maxdes}{\textit{maxdes}}
\newcommand{\op}{\mathrm{op}}
\newcommand{\wop}{\textit{wop}}

\newcommand\blockn[3]{
		\coordinate (prev) at (#2-0.5,#3+0.5);
		\draw[help lines] (#2,#3) -- +(0,1);
		\foreach \dir in {#1}{
			\draw[very thick] (prev)+(.5,-.5) rectangle +(1.5,.5);
			\path (prev)+(1,0) node {$\dir$} coordinate (prev);
		};
}
\newcommand\Dpath[4]{
	\coordinate (axis) at (0,0);
	\coordinate (axis2) at (0,0);
	\draw[help lines] (#1) grid +(#2,#3);
	\draw[dashed] (#1) -- +(#2,#3);
	\coordinate (prev) at (#1);
	
	\foreach \x in {#4}{	
		\ifnum\x=-1
		\draw[line width=2pt,blue] (axis)+(axis2) -- +(#2,0);
		\else
		\draw[line width=2pt,blue] (axis)+(\x,0) -- +(\x,1);
		\draw[line width=2pt,blue] (axis)+(axis2) -- +(\x,0);
		\path (axis) -- +(0,1) coordinate (axis);
		
		\path (axis2) -- (\x,0) coordinate (axis2);	
		\fi
	}
}
\newcommand{\filll}[3]{\node at (#1,#2) {$#3$}}

\newcommand{\fillll}[3]{\node at (#1-.5,#2-.5) {$#3$}}
\newcommand{\filllll}[2]{\node at (#1-.5,#2-.5) {$#2$}}
\newcommand{\filllllg}[2]{\node at (#1-.5,#2-.5) {\color{white!40!black}$#2$}}
\newcommand{\fillcross}[2]{\draw[thick] (#1-1,#2-1)--(#1,#2);\draw[thick] (#1-1,#2)--(#1,#2-1)}
\newcommand{\fillgcross}[2]{\draw[thick] (#1-1,#2-1)--(#1,#2);\draw[thick] (#1-1,#2)--(#1,#2-1)}
\newcommand{\thn}[1]{$#1^\textnormal{th}$}
\newcommand{\Dyck}{\mathcal{D}}
\newcommand{\fillshade}[1]{\foreach \x/\y in {#1}{\path[fill,black!15!white] (\x-1,\y-1) rectangle (\x,\y);}}

\newcommand{\lift}{\mathrm{lift}}
\newcommand{\WOPln}{\overline{\mathcal{WOP}}_n(123)}
\newcommand{\WOPlnn}[1]{\overline{\mathcal{WOP}}_{#1}(123)}
\newcommand{\filldot}[1]{\draw[fill] (#1) circle [radius=0.05];}

\title{Patterns in words of ordered set partitions}

\author{
Dun Qiu\\[-0.8ex]
\small Department of Mathematics\\[-0.8ex]
\small University of California, San Diego\\[-0.8ex]
\small La Jolla, CA 92093-0112. USA\\[-0.8ex]
\small \texttt{duqiu@ucsd.edu}
\and 
Jeffrey Remmel \\
\small Department of Mathematics\\[-0.8ex]
\small University of California, San Diego\\[-0.8ex]
\small La Jolla, CA 92093-0112. USA\\[-0.8ex]
}
\begin{document}
\maketitle
\begin{abstract}
\noindent 
An ordered set partition of $\{1,2,\ldots,n\}$ is a partition with an ordering on the parts. Let $\mathcal{OP}_{n,k}$ be the set of ordered set partitions of $[n]$ with $k$ blocks. Godbole, Goyt, Herdan and Pudwell  defined $\mathcal{OP}_{n,k}(\sigma)$ to be the set of ordered set partitions in $\mathcal{OP}_{n,k}$ avoiding a permutation pattern $\sigma$ and obtained the formula for $|\mathcal{OP}_{n,k}(\sigma)|$ when the pattern $\sigma$ is of length $2$. Later, Chen, Dai and Zhou found a formula algebraically for $|\mathcal{OP}_{n,k}(\sigma)|$ when the pattern $\sigma$ is of length $3$. 

In this paper, we define a new pattern avoidance for the set $\mathcal{OP}_{n,k}$, called $\mathcal{WOP}_{n,k}(\sigma)$, which includes the questions proposed by Godbole, Goyt, Herdan and Pudwell. We obtain formulas for $|\mathcal{WOP}_{n,k}(\sigma)|$ combinatorially for any $\sigma$ of length $ 3$. We also define 3 kinds of descent statistics on ordered set partitions and study the distribution of the descent statistics on $\mathcal{WOP}_{n,k}(\sigma)$ for $\sigma$ of length  $3$.\\
\textbf{Keywords: }permutations, ordered set partitions, pattern avoidance, bijections, Dyck paths
\end{abstract}

\section{Introduction}

In \cite{GGHP}, Godbole, Goyt, Herdan and Pudwell initiated the study of
patterns in ordered set partitions. In particular, 
they  studied the 
number of ordered set partitions which avoid certain types of 
permutations of length 2 and 3. 
A {\em partition} $\pi$ of $[n]= \{1, \ldots, n\}$ is a family of nonempty, pairwise disjoint subsets $B_1,B_2,\ldots,B_k$ of $[n]$ called {\em parts} ({\em blocks}) such that $\bigcup^k_{i=1}B_i=[n]$. We let $\ell(\pi)$ denote the number of parts in $\pi$ and 
$|\pi|=n$ denote the size of $\pi$.  
We let $\min(B_i)$ and $\max(B_i)$ denote the minimal and maximal elements 
of $B_i$ and we use the convention that we order the parts so that  
$\min(B_1)< \cdots < \min(B_k)$. To simplify notation, we 
shall write $\pi$ as $B_1/ \cdots /B_k$. Thus we would write 
$\pi =134/268/57$ for the set partition $\pi$ of $[8]$ with parts 
$B_1 = \{1,3,4\}$, $B_2 = \{2,6,8\}$ and $B_3=\{5,7\}$. 
Pattern avoidance problems in set partitions was studied by Sagan \cite{Sagan}; Jel\'inek and Mansour \cite{JM}; Jel\'inek, Mansour and Shattuck \cite{JMS}.
See Mansour \cite{M} for a
comprehensive introduction to set partitions.

An {\em ordered set partition} with underlying 
set partition $\pi$ is just a permutation of the parts of $\pi$, 
i.e.\ $\delta =B_{\sg_1}/ \cdots /B_{\sg_k}$ for some permutation $\sg$ in the 
symmetric group $S_k$ . 
For example, $\delta =57/134/268$ is an ordered set partition 
of the set $[8]$ with underlying set partition $\pi =134/268/57$. Given 
an ordered set partition $\delta =B_{\sg_1}/ \cdots /B_{\sg_k}$, we 
let the {\em word} of $\delta$, $w(\delta)$, be  the word obtained from $\delta$ by removing all the slashes.  For example, if $\delta =57/134/268$, 
then $w(\delta) = 57134268$. 
We let $\OP_n$ denote the set of ordered set partitions of 
$[n]$ and $\OP_{n,k}$ denote the set of ordered set partitions of 
$[n]$ with $k$ parts. 

If $b_1, \ldots, b_k$ are positive integers, then 
we let 
\begin{enumerate}
\item $\OP_{[b_1,\ldots,b_k]}$ denote the set of ordered set 
 partitions $B_1/ \cdots /B_k$ of $[b_1+ \cdots + b_k]$ such 
that $|B_i| = b_i$ for $i =1, \ldots, b_k$, 
\item $\OP_{n,\{b_1,\ldots,b_k\}}$ denote the set of ordered set 
 partitions $\pi \in \OP_n$ such that the size of any part in 
$\pi$ is an element of $\{b_1,\ldots,b_k\}$, 
and 
\item $\OP_{\langle b_1^{\beta_1},\ldots ,b_k^{\beta_k}\rangle}$ denote the set of ordered set 
 partitions $\pi$  
of $[\sum_{i=1}^k \beta_i b_i]$ which has $\beta_i$ parts of size 
$b_i$ for $i = 1, \ldots, k$. 
\end{enumerate}
Note that 
$$\bigcup_{n \geq 0} \OP_{n,\{b_1,\ldots,b_k\}} = \bigcup_{\beta_1 \geq 0, \ldots, \beta_k \geq 0} \OP_{\langle b_1^{\beta_1},\ldots, b_k^{\beta_k}\rangle}.$$
Clearly, $|\OP_{[b_1,\ldots,b_k]}| = \binom{n}{b_1, \ldots, b_k}$
if $b_1+ \cdots +b_k = n$. 

Given a sequence of distinct positive integers  $w= w_1 \cdots w_n$, we let 
$\red(w)$ denote the permutation in $S_n$ obtained from $w$ by 
replacing the \thn{i} smallest letter in $w$ by $i$. For  example, 
$\red(4592) =2341$. Following \cite{GGHP}, we say that  
a permutation $\sg = \sg_1 \cdots \sg_j$ {\bf occurs} in an  
ordered set partition $\delta =B_{1}/ \cdots /B_{k}$ if 
and only if  there exists $1 \leq i_1 < \cdots < i_j \leq k$ and 
$b_{i_m} \in B_{i_m}$ such that 
$\red(b_{i_1} \cdots b_{i_j}) = \sg$, and $\delta$ {\bf avoids} $\sg$ if 
$\sg$ does not occur in $\delta$.  For example, if  
$\delta =57/134/268$, then $213$ occurs in $\delta$ since 
$\red(518) = 213$, but $\delta$ avoids $123$ because every element 
in the first part $\{5,7\}$ of $\delta$ is bigger than every element in the 
second part $\{1,3,4\}$ of $\delta$. If $\alpha$ is a permutation 
in $S_j$, then we let $\OP_{n}(\alpha)$ denote the set of 
ordered set partitions of $[n]$ that avoid $\alpha$.  We can then define  
$\OP_{n,k}(\alpha)$, $\OP_{[b_1, \ldots, b_k]}(\alpha)$, 
$\OP_{n,\{b_1, \ldots, b_k\}}(\alpha)$ and 
 $\OP_{\langle b_1^{\beta_1},\ldots, b_k^{\beta_k}\rangle}(\alpha)$ in 
a similar manner. We let  
\begin{eqnarray*}
op_n(\alpha) &:=& |\OP_n(\alpha)|, \\
op_{n,k}(\alpha) &:=& |\OP_{n,k}(\alpha)|, \\
op_{[b_1, \ldots, b_k]}(\alpha) &:=& |\OP_{[b_1, \ldots, b_k]}(\alpha)|,
\mbox{ \ \ and} \\
op_{\langle b_1^{\beta_1},\ldots, b_k^{\beta_k}\rangle}(\alpha) &:=&
|\OP_{\langle b_1^{\beta_1},\ldots, b_k^{\beta_k}\rangle}(\alpha)|.
\end{eqnarray*}

Godbole, Goyt, Herdan and Pudwell \cite{GGHP} proved a number of 
interesting results about these quantities. For example, 
they showed that 
\begin{equation*}
op_{n,k}(\sg) = op_{n,k}(123)
\end{equation*}
 for 
all permutations $\sg$ of length $3$. 
They also proved that 
\begin{equation*}
op_{n,3}(123) =op_{n,3}(132) = \left( \frac{n^2}{8} +\frac{3n}{8} -2\right)2^n+3 
\end{equation*}
and 
\begin{equation*}
op_{n,n-1}(123) = \frac{3(n-1)^2\binom{2n-2}{n-1}}{n(n+1)}.
\end{equation*}
Later, Chen, Dai and Zhou \cite{CDZ} proved that 
\begin{equation}\label{CDZ}
1+ \sum_{n \geq 1}t^n \sum_{k=1}^n op_{n,k}(123) x^k =
\frac{-x+2xt-2t+2t^2x+2t^2+x\sqrt{1-4xt-4t+4t^2x+4t^2}}{2t(x+1)^2(t-1)}.
\end{equation}

The goal of this paper is to study an alternative 
notion of pattern avoidance in ordered set partitions. Given an ordered set partition 
$\delta =B_1/ \cdots /B_k$ of $[n]$, let $w(\delta) =w_1 \cdots w_n$ 
denote the word of $\delta$. Then we say that a permutation 
$\alpha = \alpha_1 \cdots \alpha_j \in S_j$ {\bf occurs 
in the word of $\delta$ } if there exists $1 \leq i_1 < \cdots < i_j \leq n$ such 
that $\red(w_{i_1} \cdots w_{i_j}) = \alpha$.  Thus $\alpha$ 
occurs in the word 
of $\delta$ if $\alpha$ classically occurs in $w(\delta)$.  We say 
that an ordered set partition $\delta$ {\bf word-avoids} $\alpha$ if 
$\alpha$ does not occur in the word of $\delta$. For example, if 
$\delta =57/134/268$, we saw that $\delta$ avoids 
$123$ in the sense of \cite{GGHP}, but clearly 
$123$ occurs in the word of $\delta$ since $\red(134) =123$. 
Then we let $\WOP_n(\alpha)$ denote the set 
of ordered set partitions which word-avoid $\alpha$.  Similarly, 
we can define $\WOP_{n,k}(\alpha)$, $\WOP_{[b_1, \ldots, b_k]}(\alpha)$, 
$\WOP_{n,\{b_1, \ldots, b_k\}}(\alpha)$ 
and $\WOP_{\langle b_1^{\beta_1},\ldots, b_k^{\beta_k}\rangle}(\alpha)$.
Then 
we let 
\begin{eqnarray*}
\wop_n(\alpha) &:=& |\WOP_n(\alpha)|, \\
\wop_{n,k}(\alpha) &:=& |\WOP_{n,k}(\alpha)|, \\
\wop_{[b_1, \ldots, b_k]}(\alpha) &:=& |\WOP_{[b_1, \ldots, b_k]}(\alpha)|,
\mbox{ \ \ and} \\ 
\wop_{\langle b_1^{\alpha_1},\ldots ,b_k^{\alpha_k}\rangle}(\alpha) &:=&
|\WOP_{\langle b_1^{\alpha_1},\ldots ,b_k^{\alpha_k}\rangle}(\alpha)|.
\end{eqnarray*}
We also study the corresponding generating functions 
\begin{eqnarray*}
\mathbb{WOP}_{\alpha}(t) &:=&  1 + \sum_{n \geq 1}  
\wop_n(\alpha)\ t^n, \\
\mathbb{WOP}_{\alpha}(x,t)&:=& 1 + \sum_{n \geq 1} t^n 
\sum_{k=1}^n \wop_{n,k}(\alpha) \ x^k, \mbox{ \ \ and} \\
\mathbb{WOP}_{\alpha,\{b_1,\ldots,b_k\}}(x,t,q_1, \ldots 
q_k)&:=& \sum_{\beta_1 \geq 0} \cdots \sum_{\beta_k \geq 0} 
\wop_{\langle b_1^{\beta_1},\ldots, b_k^{\beta_k}\rangle}(\alpha) \ 
t^{\sum_{i=1}^k b_i \beta_i} x^{\sum_{i=1}^k \beta_i} \ q_1^{\beta_1} \cdots 
q_k^{\beta_k}. 
\end{eqnarray*}

Note that $\wop_{n,k}(321) = op_{n,k}(321)$. That is, 
if $321$ occurs in the word of an ordered set partition 
$\delta$, then the occurrences of 3, 2 and 1 must have been 
in different parts of the partition $\delta$ so that 
$321$ would occur in $\delta$ in the sense of 
Godbole, Goyt, Herdan and Pudwell \cite{GGHP}. However, for other 
$\sg \in S_3$, it is not the case that 
$\wop_{n,k}(\sg) = op_{n,k}(\sg)$. In fact, it 
follows from the results of the this paper that we 
have $3$ Wilf-equivalence where $\wop_n(\sg)$ for $\sg \in S_3$, namely
 $\wop_n(123)$, $\wop_n(132) =  \wop_n(231) = \wop_n(312) = \wop_n(213)$ and $\wop_{n}(321)$.

We shall also study refinements of these generating functions 
by descents. 
Recall that for a permutation $\sg=\sg_1\cdots\sg_n\in S_n$, the {\em descent set} of $\sg$ is defined as $\Des(\sg)=\{i:\sg_i>\sg_{i+1}\}$, and the number of descents of $\sg$ is $\des(\sg)=|\Des(\sg)|$.
In fact, there are four natural notions 
of descents in an ordered set partition $\pi = B_1/\cdots /B_k\in\OP_n$. 
That is, we let $\des(\pi)$ be the number of descents 
in the word of $\pi$, $w(\pi) = w_1 \cdots w_n$.  Thus 
$\des(\pi) :=|\{i:w_i > w_{i+1}\}|$. Given 
two consecutive parts $B_i$ and $B_{i+1}$, we 
write $B_i >_{p} B_{i+1}$ if every element of $B_i$ is 
greater than every element of $B_{i+1}$ and we write 
$B_i >_{min} B_{i+1}$ if the minimal element of $B_i$ is 
greater than the minimal element of $B_{i+1}$. We shall 
call elements $i$ such that $B_i >_{p} B_{i+1}$ \emph{part-descents} 
and elements $i$ such that $B_i >_{min} B_{i+1}$ \emph{min-descents}. We also let $i$ such that $\max(B_i)>\max(B_{i+1})$ be a {\em max-descent}.
Then 
we define 
\begin{eqnarray*}
\des(\pi)&:=&|\{i: w(\pi)_i>w(\pi)_{i+1}\}|=|\{i: \max(B_i)>\min(B_{i+1})\}|,\\
\pdes(\pi) &:=& |\{i: B_i >_{p} B_{i+1}\}|=|\{i: \min(B_i)>\max(B_{i+1})\}|,\\
\mindes(\pi) &:=& |\{i: B_i >_{min} B_{i+1}\}|=|\{i: \min(B_i)>\min(B_{i+1})\}| \ \ \ \ \mbox{and} \\
\maxdes(\pi) &:=& |\{i: \max(B_i)>\max(B_{i+1})\}|.
\end{eqnarray*}
The statistics des, \pdes\ and \mindes\ are not 
equi-distributed on $\OP_n$ (as can be seen when $n=3$). 
We shall show in Section 2 that the statistics \maxdes\ and \mindes\ are 
equi-distributed on $\OP_n$. A number of other Euler-Mahonian statistics of ordered set partitions were studied in \cite{IKZ1,IKZ2,KZ,RW,S}. Wilson \cite{W} also studied Mahonian statistics of ordered multiset partitions.

For each type of generating function above, we consider the refined 
generating function where we keep track of the number of descents 
of each type. In particular, we shall study the following generating functions,
\begin{eqnarray*}
\mathbb{WOP}^{\des}_{\alpha}(x,y,t)&:=& 1 +\sum_{n \geq 1} t^n 
\sum_{\pi \in \WOP_n(\alpha)} x^{\ell(\pi)} y^{\des(\pi)}, \\
\mathbb{WOP}^{\pdes}_{\alpha}(x,y,t)&:=& 1 + \sum_{n \geq 1} t^n 
\sum_{\pi \in \WOP_n(\alpha)} x^{\ell(\pi)} y^{\pdes(\pi)}, \ \mbox{and} \\
\mathbb{WOP}^{\mindes}_{\alpha}(x,y,t)&:=& 1 + \sum_{n \geq 1} t^n 
\sum_{\pi \in \WOP_n(\alpha)} x^{\ell(\pi)} y^{\mindes(\pi)}.
\end{eqnarray*}
Similarly, we shall study
\begin{eqnarray*}
\mathbb{WOP}^{\des}_{\alpha,\{b_1,\ldots,b_k\}}(x,y,t,q_1, \ldots, q_n)&:=& \hskip -4mm 
\sum_{\beta_1 \geq 0, \ldots, \beta_k \geq 0,} 
\sum_{\pi \in \WOP_{\langle b_1^{\beta_1} ,\ldots, b_k^{\beta_k}\rangle}(\alpha)} 
t^{|\pi|} x^{\ell(\pi)} y^{\des(\pi)} q_1^{\beta_1} \cdots q_k^{\beta_k}, \\
\mathbb{WOP}^{\pdes}_{\alpha,\{b_1,\ldots,b_k\}}(x,y,t,q_1, \ldots, q_n) 
&:=&  \hskip -4mm 
\sum_{\beta_1 \geq 0, \ldots, \beta_k \geq 0,} 
\sum_{\pi \in \WOP_{\langle b_1^{\beta_1}, \ldots ,b_k^{\beta_k}\rangle}(\alpha)} 
t^{|\pi|} x^{\ell(\pi)} y^{\pdes(\pi)} q_1^{\beta_1} \cdots q_k^{\beta_k}, \\
\mathbb{WOP}^{\mindes}_{\alpha,\{b_1,\ldots,b_k\}}(x,y,t,q_1, \ldots, q_n)&:=&  \hskip -4mm 
\sum_{\beta_1 \geq 0, \ldots, \beta_k \geq 0,} 
\sum_{\pi \in \WOP_{\langle b_1^{\beta_1}, \ldots ,b_k^{\beta_k}\rangle}(\alpha)} \hskip -3mm 
t^{|\pi|} x^{\ell(\pi)} y^{\mindes(\pi)} q_1^{\beta_1} \cdots q_k^{\beta_k}.
\end{eqnarray*}

The main focus of this paper is studying the generating functions 
described above where $\alpha$ is in $S_2$ or $S_3$. 
One advantage of our notion of word-avoidance in ordered set partitions is 
that we can employ standard techniques from the theory of generating functions 
such as the Lagrange Inversion Theorem to give us nice answers. For example, 
we will show that 
\begin{equation*}
\mathbb{WOP}_{132}(x,t) = \frac{t+1-\sqrt{(t+1)^2 -4t(x+1)}}{2t(1+x)},
\end{equation*}
\begin{equation*}
\wop_{n,k}(132) = \frac{1}{k} \binom{n-1}{k-1} \binom{n+k}{k-1},
\end{equation*}
and 
\begin{equation*}
\wop_{\langle b_1^{\beta_1} ,\ldots, b_k^{\beta_k}\rangle}(132) 
= 
\frac{1}{n} \binom{k}{\beta_1, \ldots, \beta_k}\binom{n+k}{n-1},
\end{equation*}
where $n = \sum_{i=1}^k b_i \beta_i$ and $k = \sum_{i=1}^k \beta_i$. 

Similarly, we will show that 
\begin{equation*}
\mathbb{WOP}^{\des}_{132}(x,y,t) = 
\frac{(1+2yt+xyt -t -xt)-\sqrt{((1+2yt+xyt -t -xt))^2 -
		4t(1-t+ty)(x+yx)}}{2t(y+xy)}
\end{equation*}
and
\begin{equation*}
\sum_{\pi \in \WOP_{n,k}(132)} y^{\des(\pi)} = \frac{1}{k} \binom{n-1}{k-1} 
\sum_{j-0}^{k-1} \binom{k}{j}\binom{n-1}{k-1-j}y^{k-1-j}.
\end{equation*}

The outline of this paper is as follows. In Section 2, we will 
compute generating functions for ordered set partitions word-avoiding patterns of length 2 and prove some symmetries in the generating functions $\mathbb{WOP}_{\alpha}^{\des}(x,y,t)$, 
$ \mathbb{WOP}_{\alpha}^{\pdes}(x,y,t)$ and  
$\mathbb{WOP}_{\alpha}^{\mindes}(x,y,t)$ 
for $\alpha \in S_j$ and $j \geq 3$.
In Section 3, 
we will show how to compute generating functions $\mathbb{WOP}^{\des}_{\alpha}(x,y,t)$ for all $\alpha\in S_3$. In Sections 4 and 5, 
we will study generating functions $\mathbb{WOP}^{\pdes}_{\alpha}(x,y,t)$ and $\mathbb{WOP}^{\mindes}_{\alpha}(x,y,t)$ for $\alpha$ in $S_3$. 
In Section 6, we will summarize open problems about our research. 
\section{Preliminaries}

The structures of elements in $\WOP_n(12)$ and $\WOP_n(21)$ are 
quite easy to describe.  For example, if $\pi \in \WOP_n(12)$, 
then the word of $\pi$ must be $n(n-1) \cdots 21$ and hence 
$\pi = n/n-1/\cdots /1$. Similarly, if $\pi \in \WOP_n(21)$, 
then the word of $\pi$ must be $12 \cdots (n-1)n$ and hence 
$\pi$ must be of the form $B_1/B_2/ \cdots /B_k$ where 
for each $i =1, \ldots, k-1$, all the elements of $B_i$ are 
smaller than all the elements of $B_{i+1}$. It follows 
that $\wop_{n,k}(21) = \binom{n-1}{k-1}$ because to specify 
an ordered set partition $\pi\in \WOP_{n,k}(21)$ with $k$ parts, we  
only need to specify where we place the $k-1$ slashes in 
the $n-1$ spaces between the letters $1, \ldots, n$. 

Thus, 
\begin{equation*}
\mathbb{WOP}_{12}^{\des}(x,y,t) =  1+ \sum_{n \geq 1} 
y^{n-1}x^nt^n = 1+ \frac{xt}{1-xyt}
\end{equation*} 
and $\mathbb{WOP}_{12}^{\des}(x,y,t) = \mathbb{WOP}_{12}^{\pdes}(x,y,t)= 
\mathbb{WOP}_{12}^{\mindes}(x,y,t)$.
Similarly, \begin{eqnarray*}
\mathbb{WOP}_{21}^{\des}(x,y,t) &=& 1+ \sum_{n \geq 1} 
t^n \sum_{k=1}^n \binom{n-1}{k-1}x^k \nonumber \\
&=& 1+xt \sum_{n \geq 1} 
t^{n-1} \sum_{k=1}^n \binom{n-1}{k-1}x^{k-1} \nonumber \\
&=& 1+xt \sum_{n \geq 1} t^{n-1}(1+x)^{n-1} \nonumber \\
&=& 1+ \frac{xt}{1-t(1+x)},
\end{eqnarray*}
and $\mathbb{WOP}_{21}^{\des}(x,y,t) = \mathbb{WOP}_{21}^{\pdes}(x,y,t) 
= \mathbb{WOP}_{21}^{\mindes}(x,y,t)$.

Next consider the generating functions  
$\mathbb{WOP}_{\alpha}^{\des}(x,y,t)$, 
$ \mathbb{WOP}_{\alpha}^{\pdes}(x,y,t)$, and  
$\mathbb{WOP}_{\alpha}^{\mindes}(x,y,t)$ 
when $\alpha \in S_j$ for $j \geq 3$. 
There are some obvious symmetries in our situation. Recall 
that for a permutation $\sg = \sg_1 \cdots \sg_n$, the 
{\em reverse} of $\sg$ is defined by 
$\sg^r = \sg_n \cdots \sg_1$ and the {\em complement} 
of $\sg$ is defined by 
$\sg^c = (n+1-\sg_1) \cdots (n+1 -\sg_n)$. It is 
easy to see that $\des(\sg) = \des((\sg^r)^c)$. 

We can define 
reverse and complement on ordered set partitions as well. 
That is, 
suppose that $\pi = B_1/\cdots /B_k$ is an ordered set partition of $[n]$. 
Then if $B_i=\{a_1^i< a_2^i < \cdots < a_j^i\}$, we let the complement 
of $B_i$ be $B_i^c = \{(n+1 -a^i_j) < \cdots < (n+1-a^i_2) < (n+1-a^i_1)\}$. 
Then we let the reverse of $\pi$ be $\pi^r=B_k/ \cdots /B_1$ 
and the complement of $\pi$ be $\pi^c= B_1^c/ \cdots /B_k^c$. 
Thus $(\pi^r)^c = B_k^c/ \cdots /B_1^c$. 

It is easy to see 
that if $w(\pi) =w_1 \cdots w_n$, then the word of  
$(\pi^r)^c$ is $(n+1 -w_n) \cdots (n+1-w_1)= (w(\pi)^r)^c$. 
Similarly it is easy to see that 
if $B_i >_{p} B_{i+1}$, then $B_{i+1}^c >_{p} B_{i}^c$; and if $\min(B_i)>\min(B_{i+1})$, then $\max(B_{i+1}^c) > \max(B_{i}^c)$.
Thus the operation of reverse-complement shows that \maxdes\ and \mindes\ are equi-distributed on $\OP_n$, and
\begin{eqnarray*}
\sum_{\pi \in \WOP_{n,k}(\alpha)} x^{\ell(\pi)} y^{\des(\pi)} &=& 
\sum_{\pi \in \WOP_{n,k}((\alpha^r)^c)} x^{\ell(\pi)} y^{\des(\pi)}, \\
 \sum_{\pi \in \WOP_{n,k}(\alpha)} x^{\ell(\pi)} y^{\pdes(\pi)} &=& 
\sum_{\pi \in \WOP_{n,k}((\alpha^r)^c)} x^{\ell(\pi)} y^{\pdes(\pi)}, \\
\sum_{\pi \in \WOP_{n,k}(\alpha)} x^{\ell(\pi)} y^{\maxdes(\pi)} &=& 
\sum_{\pi \in \WOP_{n,k}((\alpha^r)^c)} x^{\ell(\pi)} y^{\mindes(\pi)}.
\end{eqnarray*}
This allows us to skip the computation of \maxdes\ distribution on $\WOP_{n}(\alpha)$.

It follows that for all $1 \leq b_1 < \cdots < b_s$, 
\begin{eqnarray*}
\mathbb{WOP}_{132}^*(x,y,t) &=& \mathbb{WOP}_{213}^*(x,y,t),\\
\mathbb{WOP}_{231}^*(x,y,t) &=& \mathbb{WOP}_{312}^*(x,y,t), \\
\mathbb{WOP}_{132,\{b_1, \ldots, b_s\}}^*(x,y,t,q_1,\ldots,q_s) &=& 
\mathbb{WOP}_{213, \{b_1, \ldots, b_s\}}^*(x,y,t,q_1,\ldots,q_s),\ \mbox{and} \\
\mathbb{WOP}_{231,\{b_1, \ldots, b_s\}}^*(x,y,t,q_1,\ldots,q_s) &=& 
\mathbb{WOP}_{312,\{b_1, \ldots, b_s\}}^*(x,y,t,q_1,\ldots,q_s),
\end{eqnarray*}
where $*$ is either $\des$ or $\pdes$.

Reverse-complement does not always preserve $\mindes$. 
For example, 
$$\sum_{\pi \in \WOP_3(132)} x^{\ell(\pi)} y^{\mindes(\pi)} \neq
\sum_{\pi \in \WOP_3(213)} x^{\ell(\pi)} y^{\mindes(\pi)}.$$

In general, reverse and complement by themselves do not preserve these 
generating functions.  For example, since $|\WOP_n(123)|\neq |\WOP_n(321)|$ for any $n\geq 3$, it follows that
\begin{equation*}
\mathbb{WOP}_{123}^*(x,y,t) \neq \mathbb{WOP}_{321}^*(x,y,t),
\end{equation*}
where $*$ is $\des$, $\pdes$ or $\mindes$.

Our next theorem will show that 
\begin{equation*}
\mathbb{WOP}_{312}^{\des}(x,y,t) =  
\mathbb{WOP}_{213}^{\des}(x,y,t)
\end{equation*}
 and 
\begin{equation*}\mathbb{WOP}_{312}^{\mindes}(x,y,t) =  
\mathbb{WOP}_{213}^{\mindes}(x,y,t).\end{equation*}
Thus, 
there are only three different generating functions of 
the form  $\mathbb{WOP}_{\alpha}^{\des}(x,y,t)$ for 
$\alpha \in S_3$. 
Similarly, our next theorem  will show that for all 
$1 \leq b_1 < \cdots < b_s$, 
\begin{equation*}
\mathbb{WOP}_{213,\{b_1,\ldots,b_s\}}^{\des}(x,y,t,q_1,\ldots,q_s) 
= \mathbb{WOP}_{312,\{b_1,\ldots,b_s\}}^{\des}(x,y,t,q_1,\ldots,q_s)
\end{equation*}
and 
\begin{equation*}
\mathbb{WOP}_{213,\{b_1,\ldots,b_s\}}^{\mindes}(x,y,t,q_1,\ldots,q_s) 
= \mathbb{WOP}_{312,\{b_1,\ldots,b_s\}}^{\mindes}(x,y,t,q_1,\ldots,q_s).
\end{equation*} 

\begin{theorem}\label{thm1} There is a bijection $\phi_n:\WOP_{n}(312)
\rightarrow \WOP_{n}(213)$ such that for 
all $\pi = B_1/ \cdots /B_k \in \WOP_{n}(312)$, 
$\phi_n(\pi) = C_1/ \cdots /C_k \in \WOP_{n}(213)$ 
where $|B_i| =|C_i|$ for $i=1, \ldots,k$.  
The number $1$ is in position $k$ in $w(\pi)$ if and only if $1$ is in position 
$k$ in $w(\phi_n(\pi))$, and
$\des(\pi) = \des(\phi_n(\pi))$, $\Des(w(\pi)) = \Des(w(\phi_n(\pi)))$, 
$\mindes(\pi) = \mindes(\phi_n(\pi))$.  
\end{theorem}
\begin{proof}
We shall define $\phi_n:\WOP_{n}(312) 
\rightarrow \WOP_{n}(213)$ by induction on $n$. For $1 \leq n \leq 2$, 
we let $\phi_n$ be the identity map.  Now assume 
that we have defined $\phi_k:\WOP_{k}(312) 
\rightarrow \WOP_{k}(213)$ for $k \leq n-1$. 
We classify the ordered set partitions $\pi$ in $\WOP_{n}(312)$ 
by the position of $1$ in $w(\pi)$.  

First suppose that 1 occurs in position $1$ in $w(\pi)$. 
If 1 is in a part by itself, then 
$\pi$ is of the form $1/B_2/ \cdots /B_k$ for some $k \geq 2$. 
In this case, we can subtract $1$ from each element 
in $B_2/\cdots /B_k$ to obtain an ordered set partition 
$\pi^* = B_2^*/\cdots /B_k^*$ in $\WOP_{n-1}(312)$. 
Then let $\phi_{n-1}(B_2^*/\cdots /B_k^*) = 
C_2^*/\cdots /C_k^*$ and let 
$C_2/\cdots/C_k$ be result of adding 1 to each element 
of $C_2^*/\cdots /C_k^*$. It is easy to see 
that if we let 
$\phi_n( 1/B_2/ \cdots /B_k) = 1/C_2/ \cdots / C_k$, 
then $1/C_2/ \cdots / C_k \in \WOP_{n}(213)$,  $|B_i|=|C_i|$ for 
$i =2, \ldots, k$, 
$\des( 1/B_2/ \cdots /B_k) = \des(1/C_2/ \cdots / C_k)$, $\Des(w(1/B_2/ \cdots /B_k)) = \Des(w(1/C_2/ \cdots / C_k))$,   
and  $\mindes( 1/B_2/ \cdots /B_k) = \mindes(1/C_2/ \cdots / C_k)$.
If 1 is not in  a part by itself, then 
$\pi$ is of the form $B_1/ \cdots /B_k$ where 
$1 \in B_1$ and $|B_1| \geq 2$.   
In this case, we can remove 1 from $B_1$ and subtract $1$ from each of 
the remaining element to obtain an ordered set partition 
$\pi^* = B_1^*/\cdots /B_k^*$ in $\WOP_{n-1}(312)$. 
Then let $\phi_{n-1}(B_1^*/\cdots /B_k^*) = 
C_1^*/\cdots /C_k^*$ and let 
$C_1/\cdots/C_k$ be result of adding 1 to each element 
of $C_1^*/\cdots /C_k^*$ and then adding 1 to the first part. 
Again it  is easy to see 
that if we let 
$\phi_n( B_1/ \cdots /B_k) = C_1/ \cdots / C_k$, 
then $C_1/ \cdots / C_k \in \WOP_{n}(213)$, $|B_i|=|C_i|$ for 
$i =1, \ldots, k$,  
$\des( B_1/ \cdots /B_k) = \des(C_1/ \cdots / C_k)$, $\Des(w(B_1/ \cdots /B_k)) = \Des(w(C_1/ \cdots / C_k))$,   
and  $\mindes( B_1/ \cdots /B_k) = \mindes(C_1/ \cdots / C_k)$.

Next suppose that 1 occurs in position $r$ in $w(\pi)$ where $r \geq 2$. 
Then $\pi$ must be of the form $B_1/\cdots/B_j/B_{j+1}/ \cdots /B_k$ 
where $j \geq 1$ and 1 is the first element of  part $B_{j+1}$. 
Since $w(\pi)$ is 312-avoiding, it must be the case 
all the elements of $B_1,\ldots,B_j$ are less than all 
the elements of $B_{j+1}-\{1\},B_{j+2}, \ldots, B_k$. It follows 
that $B_1/\cdots/B_j$ is a set partition of $\{2,\ldots, r\}$ 
such that $w(B_1/ \cdots/B_j)$ reduces to a 312-avoiding permutation and 
 $B_{j+1}-\{1\}/ \cdots /B_k$ is a set partition of $\{r+1, \ldots, n\}$
such that the reduction of $w(B_{j+1}/ \cdots /B_k)$ is 312-avoiding. 
Moreover, $r-1$ is a descent in $w(\pi)$ and $B_j >_{min} B_{j+1}$. 
In this case, we let $B_{j+1}^*/ \cdots /B_k^*$ be the result 
of subtracting $r-1$ from each element of $B_{j+1}, \ldots, B_k$ except the element
1 so that $B_{j+1}^*/ \cdots /B_k^*$ is an ordered set partition 
in $\WOP_{n-r+1}(312)$ whose word starts with 1. We let 
$B_1^*/ \cdots/B_j^*$ be the result of subtracting 1 from 
each element of $B_1/\cdots/B_j$ so that 
$B_1^*/\cdots/B_j^*$ is an element of $\WOP_{r-1}(312)$. 
Now let 
$\phi_{r-1}(B_1^*/ \cdots/B_j^*) = C_1/\cdots /C_j$ and 
$\phi_{n-r+1}(B_{j+1}^*/ \cdots /B_k^*) = D_1/\cdots /D_{k-j}$. 
We can then add $n-r+1$ to each element of $C_1/\cdots /C_j$ 
to produce an ordered set partition $C_1^*/\cdots /C_j^*$ 
of $\{n-r+2,\ldots,n\}$ whose word reduces to a 213-avoiding permutation 
such that $\des(\red(w(C_1^*/\cdots /C_j^*))) = \des(w(B_1/ \cdots/B_j))$, $\Des(\red(w(C_1^*/\cdots /C_j^*))) = \Des(w(B_1/ \cdots/B_j))$, 
and $\mindes(C_1^*/\cdots /C_j^*) = \mindes(B_1/\cdots/B_j)$. Then 
we let $$\phi_n(\pi) = C_1^*/\cdots/C_j^*/D_1/\cdots/D_{k-j}.$$
It is easy to see by induction that 
$\des(w(\pi)) =\des(w(\phi_n(\pi)))$, $\Des(w(\pi)) = \Des(w(\phi_n(\pi)))$ and $\mindes(\pi) =
\mindes(\phi_n(\pi))$. Moreover, by construction 
$1$ is in position $r$ in both $w(\pi)$ and $w(\phi_n(\pi))$. 
The only thing  we have to check is that $w(\phi_n(\pi))$ is 
213-avoiding, but this follows from the fact 
that  all the elements in $C_1^*/\cdots/C_j^*$ are bigger 
than all the elements in $D_1/\cdots/D_{k-j}$, 
and 
the permutations $\red(w(C_1^*/\cdots/C_j^*))$ and $w(D_1/\cdots/D_{k-j})$ are both $213$-avoiding.
%word 
%of $C_1^*/\cdots/C_j^*$ reduces to a 213-avoiding permutation 
%and the word of $D_1/\cdots/D_{k-j}$ is a $213$-avoiding permutation. 
\end{proof}

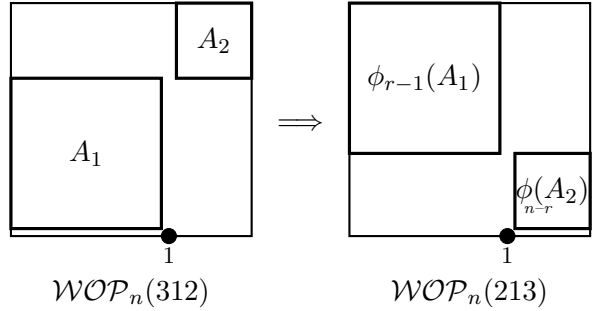
\begin{figure}[ht]
	\centering
	\vspace{-1mm}
	\begin{tikzpicture}[scale =.5]
	\draw[thick] (0,0) rectangle (6.4,6.2);
	\draw[very thick] (0,0.2) rectangle (4,4.2);
	\draw[very thick] (4.4,4.2) rectangle (6.4,6.2);
	\draw[fill] (4.2,0) circle [radius=0.2];
	\filll{2}{2.2}{A_1};\filll{5.4}{5.2}{A_2};
	\draw (4.2,-.5) node {\footnotesize $1$};
	\draw (3.2,-1.5) node {$\WOP_{n}(312)$};
	\end{tikzpicture}
	\begin{tikzpicture}[scale =.5]
	\path (0,0) rectangle (2,6.2);
	\draw (1,5.1) node {$\Longrightarrow$};
	\end{tikzpicture}	
	\begin{tikzpicture}[scale =.5]
	\draw[thick] (0,0) rectangle (6.4,6.2);
	\draw[very thick] (0,2.2) rectangle (4,6.2);
	\draw[very thick] (4.4,.2) rectangle (6.4,2.2);
	\draw[fill] (4.2,0) circle [radius=0.2];
	\filll{2}{4.2}{\phi_{r-1}(A_1)};\filll{5.4}{1.2}{\phi(A_2)};
	\node at (5,.7) {\tiny\em n--r};
	\draw (4.2,-.5) node {\footnotesize $1$};
	\draw (3.2,-1.5) node {$\WOP_{n}(213)$};
	\end{tikzpicture}
	\vspace{-2mm}
	\caption{Bijection $\phi_n:\WOP_{n}(312)
		\rightarrow \WOP_{n}(213)$.}
	\label{fig:5}
\end{figure}

\fref{6} shows that $\phi_5(3/24/15)=5/34/12$. Observe that the number of descents, word descent set, and the number of min-descents are preserved, while the number of part-descents is \emph{not} preserved.
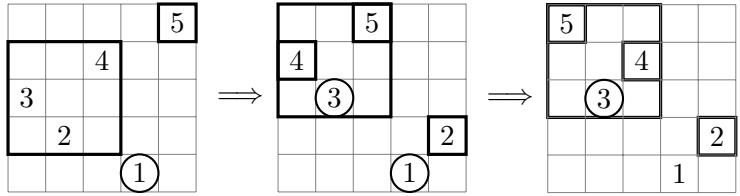
\begin{figure}[ht]
	\centering
	\vspace{-1mm}
	\begin{tikzpicture}[scale =.5]
	\draw[help lines] (0,0) grid (5,5);
	\draw[very thick] (0,1) rectangle (3,4);
	\draw[very thick] (4,4) rectangle (5,5);
	\fillll{1}{3}{3};\fillll{2}{2}{2};
	\fillll{3}{4}{4};\fillll{4}{1}{1};
	\fillll{5}{5}{5};
	\draw[thick] (3.5,.5) circle [radius=0.5];
	\end{tikzpicture}
	\begin{tikzpicture}[scale =.5]
	\draw[help lines] (0,0) grid (5,5);
	\draw[very thick] (0,2) rectangle (3,5);
	\draw[very thick] (0,3) rectangle (1,4);
	\draw[very thick] (2,4) rectangle (3,5);
	\draw[very thick] (4,1) rectangle (5,2);
	\fillll{2}{3}{3};\fillll{5}{2}{2};
	\fillll{1}{4}{4};\fillll{4}{1}{1};
	\fillll{3}{5}{5};
	\draw[thick] (3.5,.5) circle [radius=0.5];
	\draw[thick] (1.5,2.5) circle [radius=0.5];
	\draw (-1,2.5) node {$\Longrightarrow$};
	\end{tikzpicture}	
	\begin{tikzpicture}[scale =.5]
	\draw[very thick] (0,2) rectangle (3,5);
	\draw[very thick] (0,4) rectangle (1,5);
	\draw[very thick] (2,3) rectangle (3,4);
	\draw[very thick] (4,1) rectangle (5,2);
	\draw[help lines] (0,0) grid (5,5);
	\fillll{2}{3}{3};\fillll{5}{2}{2};
	\fillll{3}{4}{4};\fillll{4}{1}{1};
	\fillll{1}{5}{5};
	\draw[thick] (1.5,2.5) circle [radius=0.5];
	\draw (-1,2.5) node {$\Longrightarrow$};
	\end{tikzpicture}
	
	\caption{$\pi=3/24/15\in\WOP_{5,3}(312)\Rightarrow \phi_5(\pi)=5/34/12\in\WOP_{5,3}(213)$.}
	\label{fig:6}
\end{figure}

We end this section with two observations. 
Suppose that $\pi =B_1/\cdots/B_k \in \WOP_{n,k}(132)$. 
First, we notice that if the last element $\max(B_i)$ of $B_i$ is greater than 
the first element $\min(B_{i+1})$ of $B_{i+1}$ so that there is 
a descent in $w(\pi)$ at position $\sum_{j=1}^i |B_j|$, then 
it must be the case that 
$\min(B_i)> \min(B_{i+1})$.  That is, if $\min(B_i) < \min(B_{i+1})$, 
then $\min(B_i) \neq \max(B_i)$ and hence 
$(\min(B_i),\max(B_i),\min(B_{i+1}))$ would reduce to 132. 
It follows that for all $\pi \in \WOP_n(132)$, 
$\des(\pi) = \mindes(\pi)$, and hence, 
\begin{equation}\label{sym1}
\mathbb{WOP}^\des_{132}(x,y,t) = \mathbb{WOP}^\mindes_{132}(x,y,t).
\end{equation}

Second, for any $\pi =B_1/\cdots/B_k \in \WOP_{n,k}(132)$, $i$ is a max-descent if and only if $i$ is a part-descent. Otherwise if $\min(B_i)<\max(B_{i+1})$, then the triple $(\min(B_i),\max(B_{i}),\max(B_{i+1}))$ matches the pattern $132$. Let 
\begin{equation*}
\mathbb{WOP}^\maxdes_{132}(x,y,t):= 1 + \sum_{n \geq 1} t^n 
\sum_{\pi \in \WOP_n(132)} x^{\ell(\pi)} y^{\maxdes(\pi)},
\end{equation*}
then we have $\mathbb{WOP}^\pdes_{132}(x,y,t) = \mathbb{WOP}^\maxdes_{132}(x,y,t).$
Note that the set $\WOP_n(213)$ is in bijection with $\WOP_n(132)$ by the action of reverse-complement, and the \maxdes\ statistic on $\WOP_n(132)$ corresponds to the \mindes\ statistic on $\WOP_n(213)$. By \tref{thm1}, we have 
\begin{eqnarray} \label{sym2}
\mathbb{WOP}^\pdes_{132}(x,y,t) &=& \mathbb{WOP}^\pdes_{213}(x,y,t)
=\mathbb{WOP}^\maxdes_{132}(x,y,t)\nonumber\\
&=& \mathbb{WOP}^\mindes_{213}(x,y,t)=\mathbb{WOP}^{\mindes}_{312}(x,y,t).
\end{eqnarray}

\section{Computing $\mathbb{WOP}^{\des}_{\alpha}(x,y,t)$ for $\alpha \in S_3$}

In this section, we shall derive  generating functions 
$\mathbb{WOP}^{\des}_{\alpha}(x,y,t)$ for all $\alpha \in S_3$. 

\subsection{The functions $\mathbb{WOP}^{\des}_{132}(x,y,t)=\mathbb{WOP}^{\des}_{213}(x,y,t)=\mathbb{WOP}^{\des}_{231}(x,y,t)=\mathbb{WOP}^{\des}_{312}(x,y,t)$}
In Section 2, we have showed the equality of the four generating functions. We shall compute $\mathbb{WOP}^{\des}_{132}(x,y,t)$.
In this case, we shall classify the ordered set partitions 
$\pi$ in $\WOP_{n}(132)$ by the size of the last part. 
That is, suppose that 
$\pi = B_1/\cdots/B_k$ where $B_k =\{a_1 < \cdots < a_r\}$. 
Then we let $A_{r+1}$ denote the set of elements 
in $B_1/\cdots/B_{k-1}$ that are greater that $a_r$, 
$A_{1}$ denote the set of elements 
in $B_1/\cdots/B_{k-1}$ that are less that $a_1$, and 
$A_i$ denote the set of elements $j$ in 
$B_1/\cdots/B_{k-1}$ such that $a_i > j > a_{i-1}$ for 
$i =2, \ldots, r$.  Since $w(\pi)$ is 132-avoiding, 
for  any $i \geq 2$, every element $y$ in 
$A_i$ must appear to the left of every element $x$ in $A_{i-1}$ 
since otherwise $xya_i$ would be an occurrence of 132 in $w(\pi)$. 
It follows that the word of $\pi$ has the structure pictured 
in Figure \ref{fig:132blocks}. Note that it is possible 
that any given $A_i$ is empty. However, this structure 
ensures that no part of $\pi$ can contain elements from 
two different $A_i$'s so that if $A_i$ is non-empty, 
then $A_i$ is 
a union of consecutive parts of $\pi$, say 
$A_i = B_a/\cdots/B_b$ for some $a<b$. Moreover, if 
$i \geq 2$ and $A_i \neq \emptyset$, then the last 
element of $B_b$ is a descent in $w(\pi)$. That is, 
either $A_1, \ldots, A_{i-1}$ are empty and there 
is a descent from the last element of $B_b$ to $a_1$ 
which is the first element of $B_k$ or 
one of  $A_1, \ldots, A_{i-1}$ is non-empty. Let $p$ be the largest integer $r$ such that $1 \leq r \leq i-1$ 
and $A_r$ is non-empty, then  
there is a descent from the last element of $B_b$ to 
the first element of the first part of $A_p$.

\begin{figure}[ht]
	\centering
	\vspace{-1mm}
	\begin{tikzpicture}[scale =.5]
	\draw[thick] (0,0) rectangle (2,-2);
	\draw[thick] (2,-2) rectangle (4,-4);
	\draw[fill] (4.5,-4.5) circle [radius=0.05];
	\draw[fill] (5,-5) circle [radius=0.05];
	\draw[fill] (5.5,-5.5) circle [radius=0.05];
	\draw[thick] (6,-6) rectangle (8,-8);
	\draw[thick] (8,-8) rectangle (10,-10);
	\filll{1}{-1}{A_{r+1}};
	\filll{3}{-3}{A_r};
	\filll{7}{-7}{A_2};
	\filll{9}{-9}{A_1};
	\draw (4,-2) -- (11.5,-2) node[align=left, right] {$a_r$};
	\draw (4,-4) -- (11.2,-4) node[align=left, right] {$a_{r-1}$};
	\draw[fill] (11.6,-4.5) circle [radius=0.05];
	\draw[fill] (11.5,-5) circle [radius=0.05];
	\draw[fill] (11.4,-5.5) circle [radius=0.05];
	\draw (8,-6) -- (10.8,-6) node[align=left, right] {$a_2$};
	\draw (10,-8) -- (10.5,-8) node[align=left, right] {$a_1$};
	\draw (10,0) -- (10,-11.5);
	\filldot{4,-11}\filldot{5,-11}\filldot{6,-11}
	\filll{2}{-11}{B_1};\filll{8}{-11}{B_{k-1}};\filll{11}{-11}{B_{k}};
	\end{tikzpicture}
	\caption{The structure of $\pi \in \WOP_{n}(132)$.}
	\label{fig:132blocks}
\end{figure}
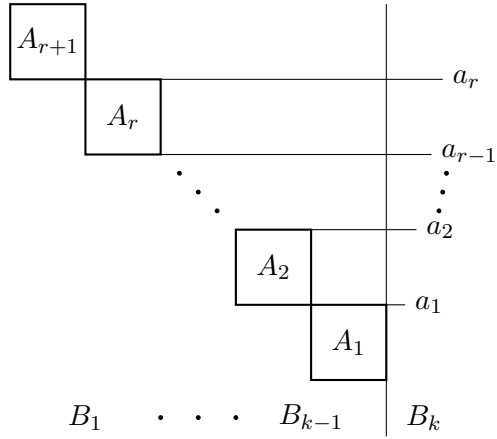

Let
$B(x,y,t) = \mathbb{WOP}^{\des}_{132}(x,y,t)$.  This structure 
implies that $B(x,y,t)$ satisfies the following 
recursion:
\begin{equation}\label{132-rec}
B(x,y,t) =1 + \sum_{r \geq 1} xt^r (1+y(B(x,y,t)-1))^r B(x,y,t).
\end{equation}
In (\ref{132-rec}) the factor $xt^r$ accounts for those 
ordered set partitions $\pi$ whose last part is of size $r$. 
We get a factor $1+y(B(x,y,t)-1)$ for $A_i$ for $i =2, \ldots, r+1$ 
where the 1 accounts for the possibility that $A_i$ is empty 
and the term $y(B(x,y,t)-1)$ accounts for the fact 
that there is descent starting at the last element of $A_i$ if 
$A_i$ is non-empty. Finally the last factor $B(x,y,t)$ corresponds to
the contribution over all possible $A_1$. 

It follows that 
\begin{equation}\label{132-rec2} 
B(x,y,t) = 1 + \frac{xtB(x,y,t)(1+y(B(x,y,t)-1))}{1-t(1+y(B(x,y,t)-1))}.
\end{equation}
Multiplying both sides of (\ref{132-rec2}) by $1-t(1+y(B(x,y,t)-1))$ 
leads to the quadratic equation
\begin{equation*}%\label{132-rec3} 
0= (1-t+ty) -B(x,y,t)(1+2yt+xyt-t-tx)+t(xy+y)B(x,y,t)^2,
\end{equation*}
and solving for $B(x,y,t)$ gives 
that 
\begin{equation*}%\label{132-rec4} 
B(x,y,t) = 
\frac{(1+2yt+xyt-t-tx) -\sqrt{(1+2yt+xyt-t-tx)^2-4(1-t+ty)(t(y+xy))}}{2t(xy+y)}.
\end{equation*}

If we let $f(x,y,t) = B(x,y,t)-1$, then (\ref{132-rec2}) gives 
that 
\begin{equation*}
f(x,y,t) = x\frac{t(f(x,y,t)+1)(1+y(f(x,y,t))}{1-t(1+yf(x,y,t))}.
\end{equation*}
The Lagrange Inversion Theorem implies that 
the coefficient of $x^k$ in $f(x,y,t)$ is given by 
$$f(x,y,t)|_{x^k}= \frac{1}{k} \delta(x)^k\Big|_{x^{k-1}},$$
where $\delta(x) =\frac{t(x+1)(1+yx)}{1-t(1+yx)}$. 
Using Newton's binomial theorem, we have
\begin{eqnarray*}
f(x,y,t)|_{x^kt^n} &=& \frac{1}{k} \frac{t^k(1+x)^k (1+yx)^k}{(1-t(1+yx))^k}\bigg|_{x^{k-1}t^n} \nonumber\\
&=&\frac{1}{k} (1+x)^k(1+yx)^k 
\left(\sum_{s \geq 0} \binom{k+s-1}{k-1} t^s(1+xy)^s\right)\Bigg|_{x^{k-1}t^{n-k}} \nonumber\\
&=& \frac{1}{k} (1+x)^k (1+yx)^n\binom{k+n-k-1}{k-1}\bigg|_{x^{k-1}} \nonumber\\
&=& \frac{1}{k}\binom{n-1}{k-1}\sum_{j=0}^{k-1} 
\binom{k}{j} \binom{n}{k-1-j}y^{k-1-j}.
\end{eqnarray*}

Thus we have the following theorem.

\begin{theorem}\label{thm:main132}
The generating function
\begin{multline*}
\mathbb{WOP}^{\des}_{132}(x,y,t) = \\
\frac{(1+2yt+xyt-t-tx) -\sqrt{(1+2yt+xyt-t-tx)^2-4(1-t+ty)(t(y+xy))}}{2t(y+yx)},
\end{multline*}
and 
\begin{equation*}
\sum_{\pi \in \WOP_{n,k}(132)} y^{\des(\pi)} = 
\frac{1}{k}\binom{n-1}{k-1}\sum_{j=0}^{k-1} 
\binom{k}{j} \binom{n}{k-1-j} y^{k-1-j}.
\end{equation*}
\end{theorem}

Setting $y=1$ in Theorem \ref{thm:main132} and 
observing that $\sum_{j=0}^{k-1} 
\binom{k}{j} \binom{n}{k-1-j} = \binom{n+k}{k-1}$, we have the following 
corollary. 

\begin{corollary}\label{cor132}
	The generating function
\begin{equation*}%\label{cormain-132}
\mathbb{WOP}_{132}(x,t) = \frac{(1+t) -\sqrt{(1+t)^2-4t(1+x)}}{2t(1+x)},
\end{equation*}
and 
\begin{equation*}
wop_{n,k}(132) = 
\frac{1}{k}\binom{n-1}{k-1}\binom{n+k}{k-1}.
\end{equation*}
\end{corollary}

It follows from Theorem \ref{thm:main132} that $\wop_n(132)$ is the number 
of rooted planar trees with $n+1$ leaves that have no vertices of 
outdegree 1 because their generating functions both satisfy the recurrence
$$
F(t)=1+\sum_{r\geq 1} t^r F(t)^{r+1}.
$$
A bijection follows naturally from the generating function:  
let $\pi = B_1/\cdots/B_k\in\WOP_n(132)$ where $B_k =\{a_1 < \cdots < a_r\}$, and $A_1,\ldots,A_{r+1}$ be the sub-ordered-partitions of $\pi$ defined by the previous construction. Then the last part $B_k$ is mapped into a root with outdegree $r+1$, and each $A_i$ is a subgraph connected to the root.
\fref{tree} shows an example of the bijection. Based on the recursion, the number of non-leaves is equal to the number of blocks of the ordered set partition, and the out-degree of the root is one more than the size of the last block.
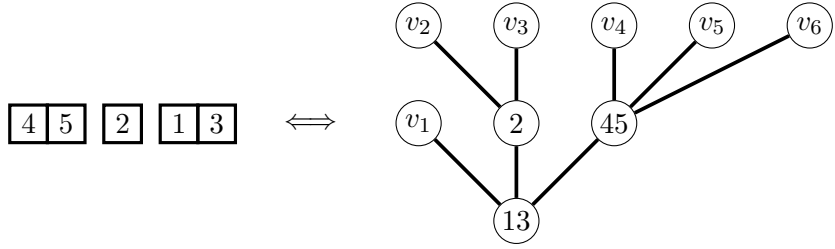
\begin{figure}[ht]
	\centering
	\vspace{-1mm}
	\begin{tikzpicture}[scale=0.5]
		\blockn{4,5}{0}{0}\blockn{2}{2.5}{0}\blockn{1,3}{4}{0}\node at (8,.5) {$\Longleftrightarrow$};\path (10,-2.7);
		\end{tikzpicture}
	\begin{tikzpicture}[>=stealth',node distance=1.3cm,on grid,initial/.style={},inner sep=0pt]
		\node[state,minimum size=.6cm] (13) {$13$};
		\node[state,minimum size=.6cm] (2) [above =of 13] {$2$};
		\node[state,minimum size=.6cm] (v1) [left =of 2] {$v_1$};
		\node[state,minimum size=.6cm] (45) [right =of 2] {$45$};
		\node[state,minimum size=.6cm] (v4) [above =of 45] {$v_4$};
		\node[state,minimum size=.6cm] (v3) [left =of v4] {$v_3$};
		\node[state,minimum size=.6cm] (v2) [left =of v3] {$v_2$};
		\node[state,minimum size=.6cm] (v5) [right  =of v4] {$v_5$};
		\node[state,minimum size=.6cm] (v6) [right =of v5] {$v_6$};
		\tikzset{mystyle/.style={-,double=black}}
		\tikzset{every node/.style={fill=white}}
		\path (13) edge [mystyle] (v1);
		\path (13) edge [mystyle] (2);
		\path (13) edge [mystyle] (45);
		\path (2) edge [mystyle] (v2);
		\path (2) edge [mystyle] (v3);
		\path (45) edge [mystyle] (v4);
		\path (45) edge [mystyle] (v5);
		\path (45) edge [mystyle] (v6);
	\end{tikzpicture}
	\caption{Bijection between $\WOP_{n}(132)$ and rooted planar trees with no vertices of outdegree $1$.}
	\label{fig:tree}
\end{figure}

Given any sequence of positive numbers $1 \leq b_1 < b_2 < \cdots < b_s$,
we let 
\begin{equation*}
A=A(x,y,t,q_1, \ldots, q_s) = 
\mathbb{WOP}^{\des}_{132,\{b_1,\ldots,b_s\}}(x,y,t,q_1, \ldots, q_s).
\end{equation*}
It follows from the block structure pictured in 
Figure \ref{fig:132blocks} that 
\begin{equation*}A = 1 +\sum_{i=1}^s x q_i t^{b_i} (1+y(A-1))^{b_i}A.\end{equation*}
If we set $F =F(x,y,t,q_1, \ldots,q_s) = 
A(x,y,t,q_1, \ldots, q_s)-1$, then 
\begin{equation*}F = x(F+1)\sum_{i=1}^s q_i t^{b_i} (1+yF)^{b_i}.\end{equation*}
It follows from Lagrange Inversion that 
\begin{equation*}F|_{x^k} = \frac{1}{k} \delta^k(x)\Big|_{x^{k-1}},\end{equation*}
where $\delta(x) = (x+1)\sum_{i=1}^s q_i t^{b_i} (1+yx)^{b_i}$.
Thus 
\begin{eqnarray}\label{thm4above}
F|_{x^kt^n} &=& \frac{1}{k}(x+1)^k 
\sum_{\overset{\alpha_i \geq 0}{\alpha_1 + \cdots + \alpha_s=k}}
\binom{k}{\alpha_1, \ldots, \alpha_s} t^{\sum_{i=1}^k\alpha_ib_i}
(1+yx)^{(\sum_{i=1}^k\alpha_ib_i)}\prod_{i=1}^s q_i^{\alpha_i}\bigg|_{x^{k-1}t^n}
\nonumber\\
&=& \frac{1}{k} (x+1)^k (1+yx)^{n}
\sum_{\overset{\alpha_1 + \cdots + \alpha_s=k}{\alpha_1b_1+ \cdots 
\alpha_kb_k=n}} 
\binom{k}{\alpha_1, \ldots, \alpha_s}
\prod_{i=1}^s q_i^{\alpha_i}\bigg|_{x^{k-1}} \nonumber\\
&=& \frac{1}{k}\left( \sum_{j=0}^{k-1} 
\binom{k}{j}\binom{n}{k-1-j}y^{k-1-j}\right)\sum_{\overset{\alpha_1 + \cdots + \alpha_s=k}{\alpha_1b_1+ \cdots 
\alpha_kb_k=n}} 
\binom{k}{\alpha_1, \ldots, \alpha_s}
\prod_{i=1}^s q_i^{\alpha_i}.
\end{eqnarray}

If $\sum \alpha_ib_i =n$, then taking the coefficient 
of $q_1^{\alpha_1}\cdots q_s^{\alpha_s}$ on both sides of equation (\ref{thm4above}) yields the following theorem. 

\begin{theorem}\label{thm:bs}
Suppose that $0 < b_1 < \cdots < b_s$, $\sum_{i=1}^s \alpha_i = k$,  and  
$\sum_{i=1}^s \alpha_i b_i  =n$.
Then 
\begin{equation*}
\sum_{\pi \in \WOP_{\langle b_1^{\alpha_1} ,\ldots,b_s^{\alpha_s}\rangle}(132)}
y^{\des(\pi)} = 
\frac{1}{k} \binom{k}{\alpha_1, \ldots, \alpha_s}\left( \sum_{j=0}^{k-1} 
\binom{k}{j}\binom{n}{k-1-j} y^{k-1-j}\right).
\end{equation*}
\end{theorem}

Setting $y=1$ in Theorem \ref{thm:bs} and observing that 
$\sum_{j=0}^{k-1} 
\binom{k}{j}\binom{n}{k-1-j} = \binom{n+k}{k-1}$ yield the following 
corollary.

\begin{corollary}\label{cor:bs}
Suppose that $0 < b_1 < \cdots < b_s$, $\sum_{i=1}^s \alpha_i = k$,  and  
$\sum_{i=1}^s {\alpha_ib_i} =n$.
Then 
\begin{equation*}\wop_{\langle b_1^{\alpha_1} ,\ldots, b_s^{\alpha_k}\rangle}(132) = 
\frac{1}{k} \binom{n+k}{k-1}\binom{k}{\alpha_1, \ldots, \alpha_s}.\end{equation*}
\end{corollary}

\subsection{The function $\mathbb{WOP}^{\des}_{123,\{1,2\}}(x,y,t,q_1,q_2)$}

Next we turn our attention to ordered set partitions $\pi$ such 
that $w(\pi)$ avoids 123.  In this case, all parts of 
$\pi$ are of size 1 or 2 since any part $B_i$ of size 
greater than 2 immediately yields a consecutive increasing 
sequence of size 3 in $w(\pi)$. 

Thus we will compute the generating function 
\begin{equation*}\mathbb{WOP}^{\des}_{123,\{1,2\}}(x,y,t,q_1,q_2):=\sum_{k,\ell\geq 0} 
\sum_{\pi \in \mathcal{WOP}_{\langle 1^k ,2^{\ell}\rangle}} y^{\des(\pi)}
t^{k+2\ell}x^{k+\ell}q_1^k q_2^{\ell}.\end{equation*}

To compute $\mathbb{WOP}^{\des}_{123,\{1,2\}}(x,y,t,q_1,q_2)$, 
we must first review a bijection $\Psi$ of Deutsch and Elizalde \cite{DE} between 123-avoiding permutations and Dyck paths. 

Given an $n\times n$ chessboard, we set the origin $(0,0)$ at the lower left corner, and label the coordinates of the columns from left to right with $0,1, \ldots, n$ 
and the coordinates of the rows from bottom to top  with $0,1 \ldots, n$.  A {\em Dyck path} is a path 
made up of unit down-steps $D$ and unit right-steps $R$ which starts at 
$(0,n)$, which is at the top left-hand corner, and 
ends at $(n,0)$, which is at the bottom right-hand corner,  and  stays weakly below the diagonal $y=n-x$. 
We let $\Dyck_n$ denote the set of Dyck paths on the $n\times n$ board.

Given a Dyck path $P$, we let 
$$Return(P) := \{1\leq i\leq n-1: P \ \mbox{goes through the point} \ (i,n-i)\}$$ 
denote the set of \emph{return positions} and let $\mathrm{return}(P) = \min(Return(P))$ be the \emph{smallest (first) return position}.  For example, for the Dyck path 
$$P =DDRDDRRRDDRDRDRRDR$$ shown on the right in  \fref{123Dn}, 
$Return(P) = \{4,8\}$ and $\mathrm{return}(P) =4$.

Given any permutation $\sg = \sg_1 \cdots \sg_n \in S_{n}(123)$, 
we write it on our  $n\times n$ chessboard 
by placing $\sg_i$ in the $i^{\mathrm{th}}$ column and $\sg_i^{\mathrm{th}}$ row, reading from bottom to top. Then, we 
shade the cells to the north-east of the cell that contains $\sg_i$. $\Psi(\sg)$ 
is the path that goes along the south-west boundary of the shaded cells. 
For example, this 
process is pictured in \fref{123Dn} for the permutation 
$\sg=869743251\in S_{9}(123)$ which is mapped  into the  
Dyck path {\em DDRDDRRRDDRDRDRRDR}.

\begin{figure}[ht]
	\centering
	\vspace{-1mm}
	\begin{tikzpicture}[scale =.5]
	\path[fill,black!15!white] (0,7) -- (1,7)--(1,5)-- (4,5)--(4,3)--(5,3)--(5,2)--(6,2)--(6,1)--(8,1)--(8,0)--(9,0)--(9,9)--(0,9);
	\draw[help lines] (0,0) grid (9,9);
	\filllll{1}{8};\filllll{2}{6};
	\filllll{3}{9};\filllll{4}{7};
	\filllll{5}{4};\filllll{6}{3};
	\filllll{7}{2};\filllll{8}{5};
	\filllll{9}{1};
	\end{tikzpicture}	
	\begin{tikzpicture}[scale =.5]
	\fillshade{1/8,2/6,3/6,4/6,5/4,6/3,7/2,9/1,8/2}
	\draw[help lines] (0,9)--(9,0);
	\draw (-2,4.5) node {$\Longrightarrow$};
	\path (-4,4.5);
	\draw[ultra thick] (0,9)--(0,7) -- (1,7)--(1,5)-- (4,5)--(4,3)--(5,3)--(5,2)--(6,2)--(6,1)--(8,1)--(8,0)--(9,0);
	\draw[help lines] (0,0) grid (9,9);
	\end{tikzpicture}
	
	\caption{$\Psi(\sg)=DDRDDRRRDDRDRDRRDR$ for $\sg=869743251$.}
	\label{fig:123Dn}
\end{figure}
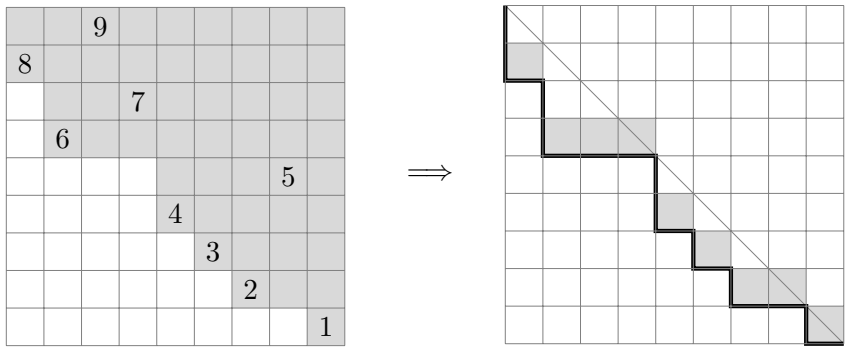

Given any Dyck path $P$, we construct the permutation $\Psi^{-1}(P)$ as follows. 
First we place a $\times$ in every outer corner of $P$. Then we consider 
the rows and columns which do not have a $\times$. Processing the columns from top to bottom 
and the rows from left to right, we place a $\times$ in the 
$i^{\mathrm{th}}$ empty row and $i^{\mathrm{th}}$ empty column. 
Finally we replace the $\times$s with numbers $\{1,\ldots,n\}$ from bottom to top.
This process is 
pictured in \fref{Dn123}. The details that $\Psi$ is bijection 
between $S_{n}(123)$ and $\mathcal{D}_n$ can be found in \cite{DE}.

\begin{figure}[ht]
	\centering
	\vspace{-1mm}
	\begin{tikzpicture}[scale =.5]
	\path[fill,black!15!white] (0,7) -- (1,7)--(1,5)-- (4,5)--(4,3)--(5,3)--(5,2)--(6,2)--(6,1)--(8,1)--(8,0)--(9,0)--(9,9)--(0,9);
	\draw[help lines] (0,0) grid (9,9);
	\draw[very thick,black!33!white] (0,8.5)--(2.5,8.5)--(2.5,5);
	\draw[very thick,black!33!white] (1,6.5)--(3.5,6.5)--(3.5,5);
	\draw[very thick,black!33!white] (4,4.5)--(7.5,4.5)--(7.5,1);
	\fillgcross{1}{8};\fillgcross{2}{6};
	\fillcross{3}{9};\fillcross{4}{7};
	\fillgcross{5}{4};\fillgcross{6}{3};
	\fillgcross{7}{2};\fillcross{8}{5};
	\fillgcross{9}{1};	
	\end{tikzpicture}	
	\begin{tikzpicture}[scale =.5]
	\fillshade{1/8,2/6,3/6,4/6,5/4,6/3,7/2,9/1,8/2}
	\draw[help lines] (0,9)--(9,0);
	\draw (-2,4.5) node {$\Longrightarrow$};
	\path (-4,4.5);
	\draw[ultra thick] (0,9)--(0,7) -- (1,7)--(1,5)-- (4,5)--(4,3)--(5,3)--(5,2)--(6,2)--(6,1)--(8,1)--(8,0)--(9,0);
	\draw[help lines] (0,0) grid (9,9);
	\filllll{1}{8};\filllll{2}{6};
	\filllll{3}{9};\filllll{4}{7};
	\filllll{5}{4};\filllll{6}{3};
	\filllll{7}{2};\filllll{8}{5};
	\filllll{9}{1};
	\end{tikzpicture}
	
	\caption{$\Psi^{-1}(P)=869743251$ for $P=DDRDDRRRDDRDRDRRDR$.}
	\label{fig:Dn123}
\end{figure}
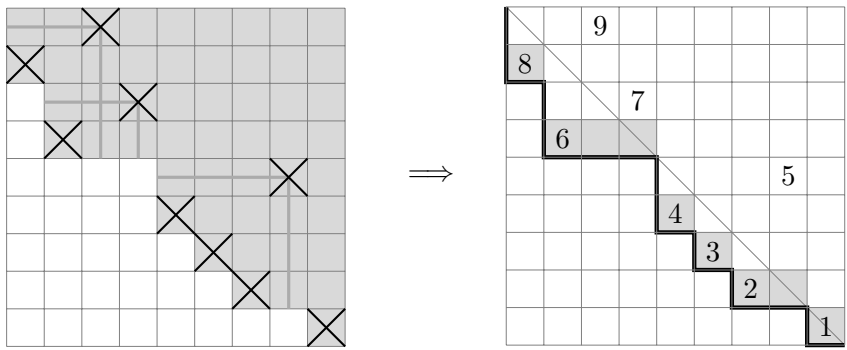

We shall classify the ordered set partitions 
$\pi \in \mathcal{WOP}_n(123)$ by the first return (from left to right)
of the path $P=\Psi(w(\pi))$. 
Suppose that 
the first return of the path $P$ is at the point 
$(n-k,k)$, then  the path $P$ is divided by the first return into $2$ paths, path $DAR$ and path $B$, as shown
in Figure \ref{fig:FirstRet} ($a$). 
The numbers in the outer corners above the point $(n-k,k)$ must come 
from $\{k+1, \ldots n\}$. Because we place the 
$\times$s in the columns which are not occupied 
by the $\times$s in the outer corners of $P$, in a decreasing manner, reading 
from left to right,  
it follows that by the time we have reached column 
$n-k$, we must have used all of the numbers in $\{k+1,\ldots, n\}$. 
This means that there is no $\times$s in the shaded area 
in \ref{fig:FirstRet} ($a$)
so that all the $\times$s in the last $k$ columns must lie 
in the lower $k$ rows.  In particular, this implies 
that in $w(\pi)$, all the elements in $\{k+1,\ldots, n\}$ proceed 
all the elements in $\{1, \ldots, k\}$. The elements in $\{k+1,\ldots, n\}$ are determined by the path $DAR$ and the elements in $\{1, \ldots, k\}$ are determined by the path $B$, and there 
is a descent
at the \thn{n-k} position 
in 
$w(\pi)$ if $k > 0$.    Hence we can  break any ordered set 
partition $\pi= B_1/ \cdots /B_j$ such that $\Psi(w(\pi)) =P$ into two parts, 
$B_1/\cdots/B_i$ that contains all the elements in $\{k+1,\ldots, n\}$ 
and $B_{i+1}/ \cdots /B_j$ that contains all the elements in 
$\{1, \ldots, k\}$. 

Let $A(x,y,t,q_1,q_2) = \mathbb{WOP}^{\des}_{123,\{1,2\}}(x,y,t,q_1,q_2)$.  It is easy to see that the contribution to $A(x,y,t,q_1,q_2)$ 
by summing over the weights of all possible choices of 
$B_{i+1}/ \cdots / B_j$ as $k$ varies over 
all choices of $k>0$ is $y(A(x,y,t,q_1,q_2)-1)$ and is equal to $1$ 
if $k=0$. 

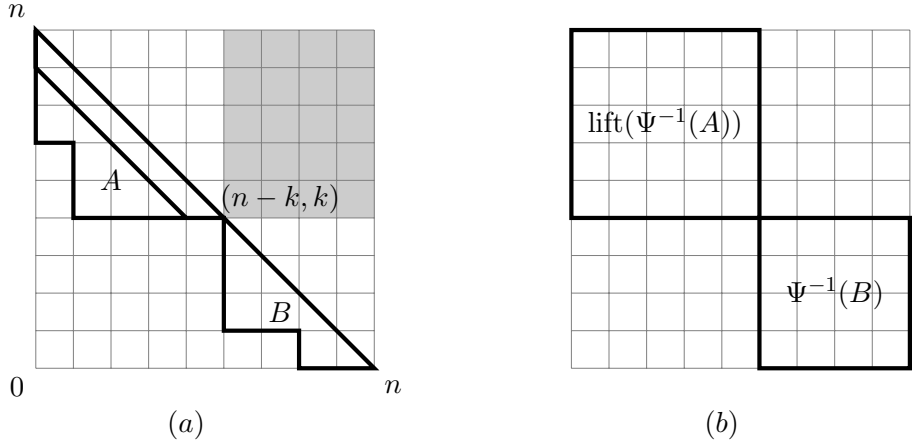
\begin{figure}[ht]
	\centering
	\vspace{-1mm}
	\begin{tikzpicture}[scale =.5]
	\fill[black!20!white] (5,4) rectangle (9,9);
	\draw[help lines] (0,0) grid (9,9);
	\draw[ultra thick] (0,9) -- (0,6) -- (1,6)--(1,4)-- (5,4)--(5,1)--(7,1)--(7,0)--(9,0)--(0,9)--(0,8)--(4,4);
	\fillll{7}{5}{(n-k,k)};
	\node at (2,5) {$A$};
	\node at (6.5,1.5) {$B$};
	\fillll{0}{0}{0};
	\fillll{10}{0}{n};
	\fillll{0}{10}{n};
	\node at (4,-1.5) {$(a)$};
	\end{tikzpicture}
	\begin{tikzpicture}[scale =.5]
	\path (-4,4);
	\draw[help lines] (0,0) grid (9,9);
	\draw[ultra thick] (0,4) rectangle (5,9);
	\draw[ultra thick] (5,0) rectangle (9,4);
	\node at (2.5,6.5) {$\mathrm{lift}(\Psi^{-1}(A))$};
	\node at (7,2) {$\Psi^{-1}(B)$};
	\node at (4,-1.5) {$(b)$};
	\path (10,0);
	\end{tikzpicture}	
	\caption{Breaking the Dyck path $P$ at the first return.}
	\label{fig:FirstRet}
\end{figure}

To analyze the contribution from parts $B_1/\cdots/B_i$, we need to work on the path $DAR$, which can be seen as lifting the path $A$ one unit in the south-west direction. We let $\lift(P)$ be the path $DPR$. For $\sg\in S_n(123)$ and $P=\Psi(\sg)$, we write $\lift(\sg)$ for the permutation $\Psi^{-1}(\lift(P))=\Psi^{-1}(DPR)\in S_{n+1}$ corresponding to path $\lift(P)$.

We say that a pair of consecutive $DR$ steps is a {\em peak} ({\em outer corner}) of a Dyck path, and in the corresponding 123-avoiding permutation, the numbers in the rows that contain peaks are called {\em peaks} of a permutation. A number is called a {\em non-peak} if it is not a peak. It is easy to see that the peaks of a permutation $\sg\in S_n(123)$ and $\lift(\sg)$ are the same. Since we label the rows and columns that do not contain peaks from left to right with the non-peak numbers
in decreasing order under the map $\Psi^{-1}$, in $\lift(\sg)$, $n+1$ is in the column of the first non-peak 
and  each remaining non-peak shifts to the next column that does not contain a peak. 
\fref{Dn123lift} illustrates the lift action of $\sg=(8,6,9,7,4,3,2,5,1)\in S_9(123)$. 

Following the construction, $\sg$ and $\lift(\sg)$ have the same descent set in the first $n-1$ positions, and there is a descent in the \thn{n} position if and only if $\sg_n$ is a non-peak. 
Since the word $w(\pi)$ of an ordered set partition $\pi\in\WOP_n(123)$ is determined by the Dyck path $DARB$, we can study smaller Dyck paths $A$ and $B$ instead of $\pi$ when computing the generating function.

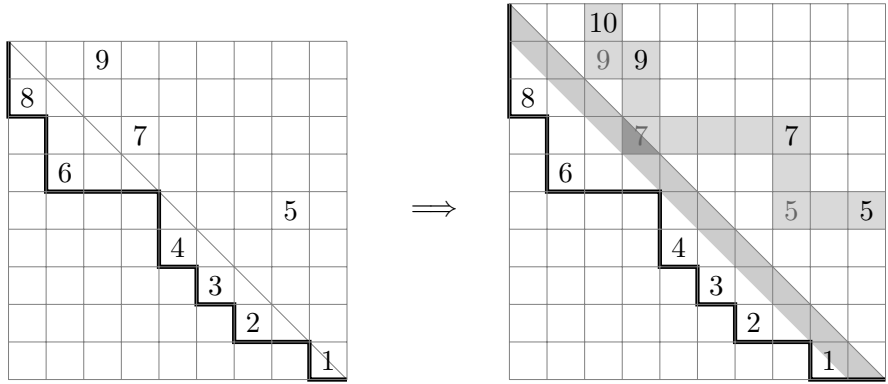
\begin{figure}[ht]
	\centering
	\vspace{-1mm}
	\begin{tikzpicture}[scale =.5]
	\draw[help lines] (0,9)--(9,0);
	\draw[ultra thick] (0,9)--(0,7) -- (1,7)--(1,5)-- (4,5)--(4,3)--(5,3)--(5,2)--(6,2)--(6,1)--(8,1)--(8,0)--(9,0);
	\draw[help lines] (0,0) grid (9,9);
	\filllll{1}{8};\filllll{2}{6};
	\filllll{3}{9};\filllll{4}{7};
	\filllll{5}{4};\filllll{6}{3};
	\filllll{7}{2};\filllll{8}{5};
	\filllll{9}{1};
	\end{tikzpicture}	
	\begin{tikzpicture}[scale =.5]
	\path[fill=black!20!white] (0,10)--(0,9) -- (9,0)--(10,0);
	\path[fill=black!15!white] (2,10)--(2,8)--(3,8)--(3,6)--(7,6)--(7,4)--(9,4)--(10,4)--
	(10,5)--(8,5)--(8,7)--(4,7)--(4,9)--(3,9)--(3,10);
	\path[fill=black!40!white] (3,6)--(3,7)--(4,6)--(3,6);
	\draw[help lines] (0,10)--(10,0);
	\draw (-2,4.5) node {$\Longrightarrow$};
	\path (-4,4.5);
	\draw[ultra thick] (0,10)--(0,7) -- (1,7)--(1,5)-- (4,5)--(4,3)--(5,3)--(5,2)--(6,2)--(6,1)--(8,1)--(8,0)--(10,0);
	\draw[help lines] (0,0) grid (10,10);
	\filllll{1}{8};\filllll{2}{6};
	\filllll{5}{4};\filllll{6}{3};
	\filllll{7}{2};
	\filllll{9}{1};
	\filllllg{3}{9};\filllllg{4}{7};\filllllg{8}{5};
	\filllll{3}{10};\filllll{4}{9};\filllll{8}{7};\filllll{10}{5};
	\end{tikzpicture}
	\caption{$\sg=(8,6,9,7,4,3,2,5,1)$ and $\mathrm{lift}(\sg)=(8,6,10,9,4,3,2,7,1,5)$.}
	\label{fig:Dn123lift}
\end{figure}

Let $\pi=B_1/\cdots/B_j\in\WOP_n(123)$ such that the first return is $n-k$ and the numbers $\{k+1,\ldots,n\}$ are contained in parts $B_1/\cdots/B_i$. We have the following 
four cases when computing the function $A(x,y,t,q_1,q_2)$. \\
\ \\
{\bf Case 1.} The first return of $P$ is at the point 
$(1,n-1)$.\\
\ \\
In this case, $P$ starts of $DR$ and $n$ is the first outer corner of path $P$. 
This means that $w(\pi)$ starts  with $n$, $i=1$, and 
$B_1 =\{n\}$.  It is easy to see that in this case the contribution to 
$A(x,y,t,q_1,q_2)$ is $xtq_1(1+y(A(x,y,t,q_1,q_2) -1))$. 
That is, if $n=1$, then we get a contribution of 
$xtq_1$ and otherwise, $n$ will cause a descent in 
$w(\pi)$ which gives a contribution of $xtq_1y(A(x,y,t,q_1,q_2) -1)$. \\
\ \\
{\bf Case 2.} The first return of $P$ is at the point 
$(2,n-2)$.\\
\ \\
In this case, $P$ starts of $DDRR$, $n-1$ is the first outer corner of 
$P$, $n$ is in the square $(2,n)$ and 
$w(\pi)$ starts out with $(n-1)n$. Then it is either 
the case that $i=2$, $B_1=\{n-1\}$, and $B_2 =\{n\}$ or 
$i=1$ and $B_1 =\{n-1,n\}$. 
It is easy to see that in the first case, the contribution to 
$A(x,y,t,q_1,q_2)$ is $x^2t^2q_1^2(1+y(A(x,y,t,q_1,q_2) -1))$. 
That is, if $n=2$, then we get a contribution of 
$x^2t^2q_1^2$ and otherwise, $n$ will cause a descent in 
$w(\pi)$ which gives a contribution of $x^2t^2q_1^2y(A(x,y,t,q_1,q_2) -1)$. Similarly, in the second case the contribution to $A(x,y,t,q_1,q_2)$ is 
$xt^2q_2(1+y(A(x,y,t,q_1,q_2)-1))$. Thus the total contribution to
$A(x,y,t,q_1,q_2)$ from Case 2 is 
\begin{equation*}
(x^2t^2q_1^2 + xt^2q_2)(1+y(A(x,y,t,q_1,q_2)-1)).
\end{equation*}
\ \\
{\bf Case 3.} The first return of $P$ is at the point 
$(n-k,k)$ where $k<n-2$, and the last three steps before the first return are $DRR$.%$k<n-2$ and $k+1$ is in column $n-k -1$.
\\
\ \\
In this case, we have the situation pictured in Figure 
\ref{fig:FirstRet2}. Thus $w(\pi) = w_1 \cdots w_n$ where 
$w_{n-k-1} = k+1$ and $w_{n-k} =p$ where $p>k+1$. It follows 
that either $B_i =\{k+1,p\}$ or $B_{i-1}=\{k+1\}$ and $B_i=\{p\}$. 
We claim that the contribution to $A(x,y,t,q_1,q_2)$ in 
the first case where $B_i =\{k+1,p\}$ is 
\begin{equation*}
y(A(x,y,t,q_1,q_2)-1)xt^2q_2(1+y(A(x,y,t,q_1,q_2)-1)).
\end{equation*}
\begin{figure}[ht]
	\centering
	\vspace{-1mm}
	\begin{tikzpicture}[scale =.5]
	\fill[black!20!white] (5,4) rectangle (9,9);
	\draw[help lines] (0,0) grid (9,9);
	\draw[ultra thick] (0,9) -- (0,6) -- (1,6)--(1,5)--(3,5)--(3,4)-- (5,4)--(5,1)--(7,1)--(7,0)--(9,0)--(0,9);
	\fillll{7}{5}{(n-k,k)};
	\fillll{0}{0}{0};
	\fillll{10}{0}{n};
	\fillll{0}{10}{n};
	\node at (3.7,4.5) {\tiny$k+1$};
	\fillcross{1}{7};\fillcross{2}{6};\fillcross{3}{9};\fillcross{5}{8};
	\end{tikzpicture}
	\caption{The situation in Case 3.}
	\label{fig:FirstRet2}
\end{figure}
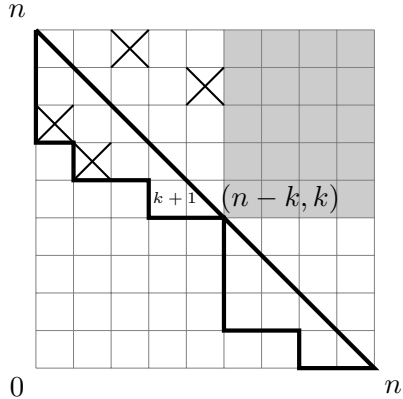

That is, the first factor $y$ comes from the fact 
that there is a descent caused by the last element of $B_{i-1}$ 
and the first element of $B_i$ which is $k+1$. The next 
factor $(A(x,y,t,q_1,q_2)-1)$ comes from summing 
 over all 
possible choices of $B_1/\cdots/B_{i-1}$. The factor 
$xt^2q_2$ comes from $B_i$. If $B_{i+1}/ \cdots /B_j$ is empty 
then we get a factor of $1$, and if $B_{i+1}/ \cdots /B_j$ is not empty, 
then we get a factor of $y$ coming from the descents between 
the last element of $B_i$ and the first element of $B_{i+1}$ and 
a factor of $(A(x,y,t,q_1,q_2)-1)$ coming summing the weights 
over all possible choices of $B_{i+1}/ \cdots /B_j$.

Similar reasoning shows that the contribution to $A(x,y,t,q_1,q_2)$ in 
the second case where $B_{i-1} =\{k+1\}$ and $B_i= \{p\}$ is 
\begin{equation*}
y(A(x,y,t,q_1,q_2)-1)x^2t^2q_1^2(1+y(A(x,y,t,q_1,q_2)-1)).
\end{equation*}

Thus the total contribution to $A(x,y,t,q_1,q_2)$ in Case 3 is 
\begin{equation*}
y(A(x,y,t,q_1,q_2)-1)(xt^2q_2+x^2t^2q_1^2)(1+y(A(x,y,t,q_1,q_2)-1)).
\end{equation*}

\noindent{\bf Case 4.} The first return of $P$ is at the point 
$(n-k,k)$ where $k<n-2$, and the last three steps before the first return are $RRR$. %and $k+1$ is in column $r$ where $r < n-k -1$.
\\
\ \\
In this case, we have the situation pictured in Figure 
\ref{fig:FirstRet3}. Thus $w(\pi) = w_1 \cdots w_n$ where 
$w_{r} = k+1$ and $w_{r+1} \cdots w_{n-k}$ is a decreasing 
sequence of length at least 2. In this situation, 
$B_i$ must be a singleton part $\{w_{n-k}\}$. 
We claim that the contribution to $A(x,y,t,q_1,q_2)$ from 
the ordered set partitions in Case 4 is 
\begin{equation*}
xytq_1(A(x,y,t,q_1,q_2) - 1 -xtq_1 -xytq_1(A(x,y,t,q_1,q_2) -1))  
(1+y(A(x,y,t,q_1,q_2)-1)).
\end{equation*}
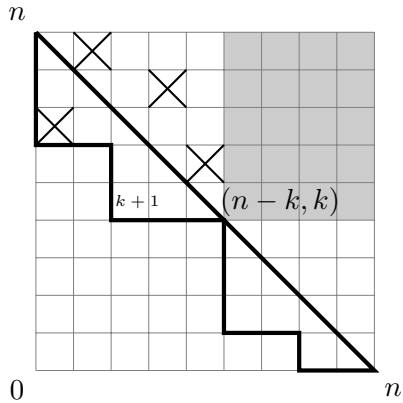
\begin{figure}[ht]
	\centering
	\vspace{-1mm}
	\begin{tikzpicture}[scale =.5]
	\fill[black!20!white] (5,4) rectangle (9,9);
	\draw[help lines] (0,0) grid (9,9);
	\draw[ultra thick] (0,9) -- (0,6) -- (2,6)--(2,4)-- (5,4)--(5,1)--(7,1)--(7,0)--(9,0)--(0,9);
	\fillll{7}{5}{(n-k,k)};
	\fillll{0}{0}{0};
	\fillll{10}{0}{n};
	\fillll{0}{10}{n};
	\node at (2.7,4.5) {\tiny$k+1$};
	\fillcross{1}{7};\fillcross{5}{6};\fillcross{2}{9};\fillcross{4}{8};
	\end{tikzpicture}
	\caption{The situation in Case 4.}
	\label{fig:FirstRet3}
\end{figure}

That is, the first factor $xytq_1$ is the weight of part $B_i$, where $y$ comes from the fact 
that there is a descent caused by the last element of $B_{i-1}$ 
and the first element of $B_i$.  The next 
factor comes summing over all possible choices 
of $B_1/ \cdots /B_{i-1}$. It is not difficult to see 
that this corresponds to the sum of the weights over 
all non-empty ordered set partitions $\pi$ where $1$ is not the 
last element of the word of $\pi$. Let 
\begin{equation*}
A_n(x,y,q_1,q_2): = \sum_{\pi \in \mathcal{WOP}_n(123)} 
x^{\ell(\pi)}y^{\des(\pi)} q_1^{\mathrm{one}(\pi)}
q_2^{\mathrm{two}(\pi)},
\end{equation*}
where $\mathrm{one}(\pi)$ is the number of parts of size 1 and $\mathrm{two}(\pi)$ is the number of parts of size 2 in 
$\pi$. Then
$A_n(x,y,q_1,q_2)-xytq_1A_{n-1}(x,y,q_1,q_2)$ is the weight over 
all ordered set partitions $\pi$ of size $n$ such that 1 is not 
the last element of $w(\pi)$. Thus the sum of the weights over 
all non-empty ordered set partitions $\pi$ where $1$ is not the 
last element of $w(\pi)$ equals 
\begin{multline*}
\sum_{n \geq 2} t^n (A_n(x,y,q_1,q_2) -xytq_1A_{n-1}(x,y,q_1,q_2)) =\\
(A(x,y,t,q_1,q_2)-1 -x tq_1 
)-xytq_1(A(x,y,t,q_1,q_2)-1).
\end{multline*}

Finally we get a factor of 1 if $B_{i+1}/\cdots/B_j$ is empty 
and a factor of $y(A(x,y,t,q_1,q_2)-1)$ over all possible 
choices of $B_{i+1}/\cdots/B_j$ if 
$B_{i+1}/\cdots/B_j$ is non-empty.

Summing the contributions from Cases 1--4, 
we have 
\begin{eqnarray}\label{eq:123y}
A(x,y,t,q_1,q_2) &=& 1 + (y-1)^2 (q_1xt+q_2xt^2-q_1^2 x^2 t^2(y-1)) - 
\nonumber \\
&&2A(x,y,t,q_1,q_2)(y(y-1)(q_1xt+q_2xt^2-q_1^2x^2 t^2(y-1))+ \nonumber \\
&&A(x,y,t,q_1,q_2)^2 y^2(q_1xt+q_2xt^2-q_1^2x^2 t^2(y-1)).
\end{eqnarray}

Because (\ref{eq:123y}) involves both linear and quadratic terms 
in $x$, we can not apply the Lagrange Inversion Theorem 
to get an explicit formula for 
$\mathbb{WOP}^{\des}_{123,\{1,2\}}(x,y,t,q_1,q_2)|_{x^k}$. Nevertheless, 
(\ref{eq:123y}) gives us a quadratic equation which we can solve for 
$A(x,y,t,q_1,q_2)$ to prove the following theorem. 

\begin{theorem}\label{W123}
	The generating function
\begin{equation*}
\mathbb{WOP}^{\des}_{123,\{1,2\}}(x,y,t,q_1,q_2) =
\frac{P(x,y,t,q_1,q_2)-\sqrt{Q(x,y,t,q_1,q_2)}}{R(x,y,t,q_1,q_2)},
\end{equation*}
where 
\begin{eqnarray*}
P(x,y,t,q_1,q_2) &=& 1+2y(y-1)q_1xt+2y(y-1)q_2xt^2-2y(y-1)^2 q_1^2x^2t^2, \\
Q(x,y,t,q_1,q_2) &=& 1-4yq_1xt-4yq_2xt^2+4(y(y-1)q_1^2x^2t^2, \ \mbox{and} \\
R(x,y,t,q_1,q_2) &=&2y^2q_1xt +2y^2q_2xt^2-2y^2(y-1)q_1^2x^2t^2.
\end{eqnarray*}
\end{theorem}

Setting $y=1$ in $\mathbb{WOP}^{\des}_{123,\{1,2\}}(x,y,t,q_1,q_2)$ gives 
us the following corollary. 

\begin{corollary}\label{123fomula}We have
	\begin{equation}\label{123eq1}
	\mathbb{WOP}_{123,\{1,2\}}(x,t,q_1,q_2)
	=\frac{1-\sqrt{1-4tx(q_1+xq_2)}}{2tx(q_1+xq_2)},
	\end{equation}
	and
	\begin{equation}\label{123eq2}
	\wop_{\langle1^{\alpha_1},2^{\alpha_2}\rangle}(123)=\frac{1}{\alpha_1+\alpha_2+1}\binom{2\alpha_1+2\alpha_2}{\alpha_1+\alpha_2}\binom{\alpha_1+\alpha_2}{\alpha_1},\end{equation}
	\begin{equation}\label{123eq3}
	\wop_{n,k}(123)=\wop_{\langle1^{2k-n},2^{n-k}\rangle}(123)=\frac{1}{k+1}\binom{2k}{k}\binom{k}{n-k}.\end{equation}
\end{corollary}
\begin{proof}
	Let $A_{123}(x,t,q_1,q_2)=\mathbb{WOP}^{\des}_{123,\{1,2\}}(x,1,t,q_1,q_2)=\mathbb{WOP}_{123,\{1,2\}}(x,t,q_1,q_2)$, then the recursion becomes
	$$
	A_{123}(x,t,q_1,q_2) = 1+tq_1 x A_{123}^2(x,t,q_1,q_2)+tq_2 x^2 A_{123}^2(x,t,q_1,q_2).
	$$
	Equation (\ref{123eq1}) is obtained by solving the quadratic equation.
	Since 
	$$
	\wop_{n,k}(123)=\wop_{\langle1^{2k-n},2^{n-k}\rangle}(123)=A_{123}(x,t,q_1,q_2)|_{t^nx^kq_1^{2k-n}q_2^{n-k}},
	$$
	we can get equation (\ref{123eq2}) and equation (\ref{123eq3}) by applying Lagrange Inversion.
\end{proof}

Thus, we have enumerated the number of ordered set partitions in $\WOP_{n}(123)$ with certain numbers of blocks of size $1$ and size $2$. Now we give a formula for the number of ordered set partitions in $\WOP_{n}(123)$ with a certain block size composition. In \cite{GGHP}, Godbole, \textit{et al.}\ showed that 
$$\op_{[b_1,\ldots,b_i,b_{i+1},\ldots,b_k]}(321)=\op_{[b_1,\ldots,b_{i+1},b_i\ldots,b_k]}(321)$$
by constructing a bijective map between $\OP_{[b_1,\ldots,b_i,b_{i+1},\ldots,b_k]}(321)$ and $\OP_{[b_1,\ldots,b_{i+1},b_i\ldots,b_k]}(321)$. 

For our new definition of pattern avoidance, we prove a similar result that the order of block sizes in block size composition does not affect $\wop_{[b_1,\ldots,b_k]}(123)$, and we have the following theorem.

\ \begin{theorem}\label{123bijn}
	We have
	$$
	\wop_{[b_1,\ldots,b_i,b_{i+1},\ldots,b_k]}(123)=\wop_{[b_1,\ldots,b_{i+1},b_i,\ldots,b_k]}(123)
	$$
	and
	$$
	\wop_{[b_1,\ldots,b_i,b_{i+1},\ldots,b_k]}(321)=\wop_{[b_1,\ldots,b_{i+1},b_i,\ldots,b_k]}(321).
	$$
\end{theorem}

\begin{proof}
	The second equation is included in the bijection constructed by Godbole, \textit{et al.}\ that
	\begin{eqnarray*}
		\wop_{[b_1,\ldots,b_i,b_{i+1},\ldots,b_k]}(321)&=&\op_{[b_1,\ldots,b_i,b_{i+1},\ldots,b_k]}(321)\nonumber\\
		&=&\op_{[b_1,\ldots,b_{i+1},b_i,\ldots,b_k]}(321)\nonumber
		\\&=&\wop_{[b_1,\ldots,b_{i+1},b_i,\ldots,b_k]}(321).
	\end{eqnarray*}
	For the first equation, we prove by a bijection.
	
	For a block size composition $B=[b_1,\ldots,b_i,b_{i+1},\ldots,b_k]$, since we are considering the $123$-avoiding ordered set partitions, all the blocks are of size $1$ or $2$. We have the following $2$ cases.
	\begin{enumerate}[(1)]
		\item If $b_i=b_{i+1}=1$ or $2$, then $\wop_{[b_1,\ldots,b_i,b_{i+1},\ldots,b_k]}(123)$ and $\wop_{[b_1,\ldots,b_{i+1},b_i,\ldots,b_k]}(123)$ are exactly the same enumeration.
		\item If $b_i\neq b_{i+1}$, then without loss of generality, we suppose $b_i=1$ and $b_{i+1}=2$. We show that there is a bijective map between $\WOP_{[b_1,\ldots,1,2,\ldots,b_k]}(123)$ and $\WOP_{[b_1,\ldots,2,1,\ldots,b_k]}(123)$. We suppose the $3$ integers filled in blocks $b_i$ and $b_{i+1}$ are $a_1<a_2<a_3$. Since there is no $123$ pattern-occurrence, there are only $2$ possible fillings for both $[\ldots,1,2,\ldots]$ and $[\ldots,2,1,\ldots]$ cases. They are $a_2/a_1a_3$ and $a_3/a_1a_2$ for $[\ldots,1,2,\ldots]$, $a_2a_3/a_1$ and $a_1a_3/a_2$ for $[\ldots,2,1,\ldots]$. We construct a map, as showed in \fref{11}, sending $a_3/a_1a_2$ to $a_1a_3/a_2$ and $a_2/a_1a_3$ to $a_2a_3/a_1$.
		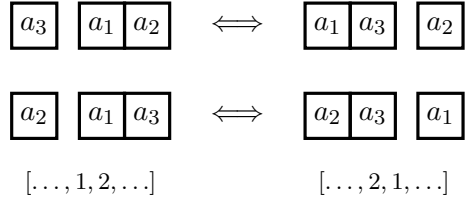
\begin{figure}[ht]
			\centering
			\vspace{-1mm}
			\begin{tikzpicture}[scale=0.6]
			\blockn{a_2}{0}{0}\blockn{a_1,a_3}{1.5}{0}
			\blockn{a_3}{0}{2}\blockn{a_1,a_2}{1.5}{2}
			\draw (5,2.5) node {$\Longleftrightarrow$};
			\draw (5,.5) node {$\Longleftrightarrow$};
			\blockn{a_1}{9}{0}\blockn{a_2,a_3}{6.5}{0}			\blockn{a_2}{9}{2}\blockn{a_1,a_3}{6.5}{2}
			\draw (1.75,-1) node {\footnotesize $[\ldots,1,2,\ldots]$};
			\draw (8.25,-1) node {\footnotesize $[\ldots,2,1,\ldots]$};
			\end{tikzpicture}
			\caption{Bijection between $\WOP_{[b_1,\ldots,1,2,\ldots,b_k]}(123)$ and $\WOP_{[b_1,\ldots,2,1,\ldots,b_k]}(123)$.}
			\label{fig:11}
			\vspace{-4mm}
		\end{figure}
	\end{enumerate}
	It is not difficult to check that the map is bijective and preserves the $123$-avoiding condition. Thus $\wop_{[b_1,\ldots,b_i,b_{i+1},\ldots,b_k]}(123)=\wop_{[b_1,\ldots,b_{i+1},b_i\ldots,b_k]}(123).$
\end{proof}

The formula for $\wop_{[b_1,\ldots,b_k]}(123)$ follows the bijection.
\begin{corollary}\label{t15}
	For any composition $[b_1,\ldots,b_k]$ such that $b_i\in\{1,2\}$, we have
	$$
	\wop_{[b_1,\ldots,b_k]}(123)=C_k,
	$$
	here $C_k= \frac{1}{k+1}\binom{2k}{k}$ is the \thn{k} Catalan number.
\end{corollary}

\begin{proof}
	Let $\alpha_1$ be the number of $1$'s and $\alpha_2$ be the number of $2$'s in $[b_1,\ldots,b_k]$. By Corollary \ref{123fomula}, we have 
	\begin{equation*}
	\wop_{\langle1^{\alpha_1},2^{\alpha_2}\rangle}(123)=\frac{1}{\alpha_1+\alpha_2+1}\binom{2\alpha_1+2\alpha_2}{\alpha_1+\alpha_2}\binom{\alpha_1+\alpha_2}{\alpha_1}.
	\end{equation*}
	Since the order of block sizes does not affect $\wop_{[b_1,\ldots,b_k]}(123)$ and there are $\binom{\alpha_1+\alpha_2}{\alpha_1}$ ways to permute the block sizes, we have
	$$
	\wop_{[b_1,\ldots,b_k]}(123)=\frac{\frac{1}{\alpha_1+\alpha_2+1}\binom{2\alpha_1+2\alpha_2}{\alpha_1+\alpha_2}\binom{\alpha_1+\alpha_2}{\alpha_1}}{\binom{\alpha_1+\alpha_2}{\alpha_1}}=\frac{1}{\alpha_1+\alpha_2+1}\binom{2\alpha_1+2\alpha_2}{\alpha_1+\alpha_2}=\frac{1}{k+1}\binom{2k}{k}=C_k.\qedhere
	$$
\end{proof}

Setting $y=q_1=q_2 =1$ in $\mathbb{WOP}^{\des}_{123,\{1,2\}}(x,y,t,q_1,q_2)$ gives 
us the following corollary. 

\begin{corollary}We have
\begin{equation*}
\mathbb{WOP}_{123}(x,t)= \frac{1-\sqrt{1-4tx-4t^2x}}{2(xt+xt^2)}.
\end{equation*}
\end{corollary}

We pause to make some observations about some special 
cases of elements of $\WOP_n(123)$.  First consider 
the case of ordered set partitions in $\mathcal{WOP}_{n}(123)$ where 
every part has size 1. In this case, we are just considering the generating 
function of $y^{\des(\sg)}$ over all 123-avoiding permutations. 
We can obtain this generating function from 
$\mathbb{WOP}^{\des}_{123,\{1,2\}}(x,y,t,q_1,q_2)$ by setting 
$x$ equal to $1/x$, $t$ equal to $tx$, and then setting $x=0$. We carried 
out these steps in Mathematica and obtained the following corollary 
which was first proved by Barnabei, Bonetti and Silimbani 
\cite{BBS}. 

\begin{corollary}We have
\begin{equation*}1+ \sum_{n \geq 1} t^n \sum_{\sg \in S_n(123)} y^{\des(\sg)} = 
\frac{-1-2ty(y-1)+2t^2y(y-1)^2 +\sqrt{1-4ty-4t^2y(y-1)}}{2ty^2(-1+t(y-1))}.
\end{equation*}
\end{corollary}

We can do a similar computation starting with the 
generating function $\mathbb{WOP}^{\des}_{132}(x,y,t)$ to obtain the 
following corollary. 
\begin{corollary} For any $\alpha \in \{132,231,312,213\}$,
\begin{equation*}1+ \sum_{n \geq 1} t^n \sum_{\sg \in S_n(\alpha)} y^{\des(\sg)} = 
\frac{1+t(y-1) -\sqrt{1+t^2(y-1)^2-2t(y+1)}}{2yt}.
\end{equation*}
\end{corollary}

In this case, the coefficients are the 
coefficients of the  triangle of the Narayana numbers 
$T(n,k) = \frac{1}{k}\binom{n}{k-1}\binom{n-1}{k-1}$ which 
is entry A001263 in the OEIS \cite{oeis}.

\subsection{The function $\mathbb{WOP}^{\des}_{321}(x,y,t)$}

The final generating function that we shall consider in this section 
is $\mathbb{WOP}^{\des}_{321}(x,y,t)$.  Since a permutation 
$\sg$ is 321-avoiding if and only if its reverse $\sg^r$ is 
123-avoiding, we shall again appeal to 
the bijection $\Psi$ of Deutsch and Elizalde between 123-avoiding permutations 
and Dyck paths and classify the ordered set partitions 
$\delta$ which word-avoid 321 by $\Psi(w(\delta))$.
The main difference in this case is that we obtain the 
permutation $w(\delta)$  by reading the elements in the diagram 
from right to left, rather from left to right, and we classify the 
ordered set partitions by the last return of 
$\Psi(w(\delta))$.  In this situation, we have 
two cases for any $\delta \in \mathcal{WOP}_n(321)$. \\
\ \\
{\bf Case 1.} The last return of $\Psi(w(\delta))$  
is at position $(n-1,1)$ in which case $w(\delta)$ starts with 1.

In this case, 1 can not be part of an occurrence of 321 in 
the word of the ordered set partition.  Thus either 1 is 
in a part by itself in which case we get a contribution 
of $xt \mathbb{WOP}^{\des}_{321}(x,y,t)$ to $\mathbb{WOP}^{\des}_{321}(x,y,t)$, 
or $1$ is part of the first part of the ordered set partition 
arising from the part of the ordered set partition above and to 
the left of 1 which gives a contribution of 
$t(\mathbb{WOP}^{\des}_{321}(x,y,t)-1)$ to $\mathbb{WOP}^{\des}_{321}(x,y,t)$. Thus 
the total contribution to $\mathbb{WOP}^{\des}_{321}(x,y,t)$ 
of the ordered set partitions that
word-avoid 321 and start with 1 is 
\begin{equation*}xt\mathbb{WOP}^{\des}_{321}(x,y,t)+t(\mathbb{WOP}^{\des}_{321}(x,y,t)-1).\end{equation*}
\ \\
{\bf Case 2.} Either $\Psi(w(\delta))$  
has  no return or the last return is 
at position $(n-k,k)$ where $k > 1$. 

Let us first consider the cases of 
ordered set partitions $\delta \in \mathcal{WOP}_n(321)$ such that 
$\Psi(w(\delta))$ hits the diagonal only at $(0,n)$ and $(n,0)$ and 
$n \geq 2$. 
For such ordered set partitions, we have two subcases. \\
\ \\
{\bf Subcase 2.1} The second element of $w(\delta)$ equals 1. \\
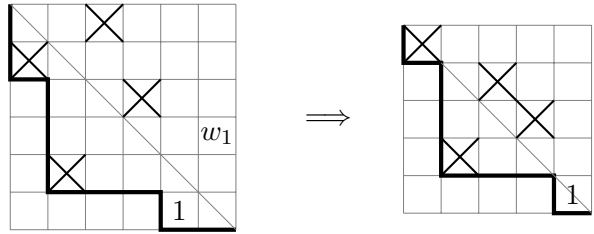
\begin{figure}[ht]
	\centering
	\vspace{-1mm}
	\begin{tikzpicture}[scale =.5]
	\draw[help lines] (0,0) grid (6,6);
	\draw[ultra thick] (0,6) -- (0,4) -- (1,4)--(1,1)--(4,1)--(4,0)--(6,0);
	\draw[help lines] (6,0)--(0,6);
	\fillll{5}{1}{1};
	\fillll{6}{3}{w_1};
	\fillcross{1}{5};\fillcross{2}{2};\fillcross{3}{6};\fillcross{4}{4};
	\end{tikzpicture}
	\begin{tikzpicture}[scale =.5]
	\draw (-2,2.5) node {$\Longrightarrow$};
	\path (-4,-.5);
	\draw[help lines] (0,0) grid (5,5);
	\draw[ultra thick] (0,5) -- (0,4) -- (1,4)--(1,1)--(4,1)--(4,0)--(5,0);
	\draw[help lines] (5,0)--(0,5);
	\fillll{5}{1}{1};
	\fillcross{1}{5};\fillcross{2}{2};\fillcross{3}{4};\fillcross{4}{3};
	\end{tikzpicture}	
	\caption{Ordered set partitions in Subcase 2.1.}
	\label{fig:Subcase21}
\end{figure}

In this case, suppose that $w(\delta) = w_1 \cdots w_n$ where 
$w_2 =1$. Then we have the situation pictured in Figure 
\ref{fig:Subcase21}. Since $w_1 > w_2 =1$, it must 
be the case that $w_1$ is in a part by itself 
so that it contributes a factor of 
$xyt$ to the weight of $\delta$.  If we remove 
the row and column containing $w_1$ and keep the 
same outer corner squares, and  possibly relabel 
the $\times$s in the columns with no outer corner 
squares by having the $\times$s in those columns decreasingly, 
reading from left to right, we will obtain an 
arbitrary ordered set partition $\pi \in \mathcal{WOP}_{n-1}(321)$ 
such that $w(\pi)$ starts with 1. Hence the ordered set 
partitions in this subcase contribute to $\mathbb{WOP}^{\des}_{321}(x,y,t)$ a factor 
of 
\begin{equation*}
xyt(xt\mathbb{WOP}^{\des}_{321}(x,y,t) +
	t(\mathbb{WOP}^{\des}_{321}(x,y,t)-1)).
\end{equation*}
\ \\
{\bf Subcase 2.2} The second element of $w(\delta)$ does not equal 1. \\
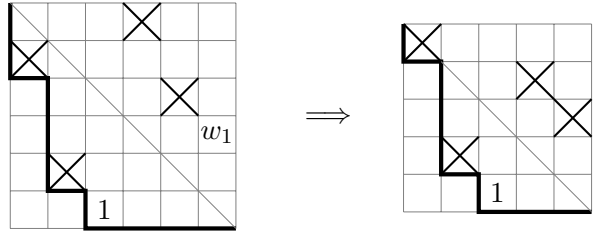
\begin{figure}[ht]
	\centering
	\vspace{-1mm}
	\begin{tikzpicture}[scale =.5]
	\draw[help lines] (0,0) grid (6,6);
	\draw[ultra thick] (0,6) -- (0,4) -- (1,4)--(1,1)--(2,1)--(2,0)--(6,0);
	\draw[help lines] (6,0)--(0,6);
	\fillll{3}{1}{1};
	\fillll{6}{3}{w_1};
	\fillcross{1}{5};\fillcross{2}{2};\fillcross{4}{6};\fillcross{5}{4};
	\end{tikzpicture}
	\begin{tikzpicture}[scale =.5]
	\draw (-2,2.5) node {$\Longrightarrow$};
	\path (-4,-.5);
	\draw[help lines] (0,0) grid (5,5);
	\draw[ultra thick] (0,5) -- (0,4) -- (1,4)--(1,1)--(2,1)--(2,0)--(5,0);
	\draw[help lines] (5,0)--(0,5);
	\fillll{3}{1}{1};
	\fillcross{1}{5};\fillcross{2}{2};\fillcross{4}{4};\fillcross{5}{3};
	\end{tikzpicture}	
	\caption{Ordered set partitions in Subcase 2.2.}
	\label{fig:Subcase22}
\end{figure}

In this case, suppose that $w(\delta) = w_1 \cdots w_n$ where 
$w_i =1$ for $i > 2$. Then we have the situation pictured in Figure 
\ref{fig:Subcase22}. In this case, since $w_1 < w_2 < \cdots <  
w_{i-1}> w_i=1$, it must 
be the case that $w_i$ starts a new part in $\delta$.  
If we remove 
the row and column containing $w_1$ and keep the 
same outer corner squares, and possibly relabel 
the $\times$s in the columns with no outer corner 
squares by having the $\times$s in those columns decreasingly, 
reading from left to right, we will obtain an 
arbitrary ordered set partition $\pi \in \mathcal{WOP}_{n-1}(321)$ 
such that $w(\pi)$ does not start  with 1. The sum of 
the weights of the ordered set partitions 
$\pi$ such that $w(\pi)$ does  not start with $1$ is 
\begin{equation*}
\mathbb{WOP}^{\des}_{321}(x,y,t) -1 - xt 
	\mathbb{WOP}^{\des}_{321}(x,y,t) -t(\mathbb{WOP}^{\des}_{321}(x,y,t)-1).
\end{equation*}
Then $w_1$ is either in a part by itself in which case 
it contributes a factor of $xt$ or is in the same part with 
$w_2$ in which case it contributes a factor of $t$. 
Hence the ordered set 
partitions in this subcase contribute a factor 
of 
\begin{equation*}
(xt+t)(\mathbb{WOP}^{\des}_{321}(x,y,t) -1 - xt 
\mathbb{WOP}^{\des}_{321}(x,y,t) -t(\mathbb{WOP}^{\des}_{321}(x,y,t)-1))
\end{equation*}
to $\mathbb{WOP}^{\des}_{321}(x,y,t)$.\\

Let 
\begin{equation*}
NR(x,y,t):= \sum_{n \geq 2} t^n 
\sum_{\substack{\pi \in \WOP_n(321),\\Return(\Psi(w(\delta)))=\emptyset}} x^{\ell(\pi)} y^{\des(\pi)} \\
\end{equation*}
be the contribution  of ordered set partitions in Subcases 2.1 and 2.2 to  $\mathbb{WOP}^{\des}_{321}(x,y,t)$, then
\begin{multline*}%\label{D}
NR(x,y,t) = xyt(xt\mathbb{WOP}^{\des}_{321}(x,y,t) +t(\mathbb{WOP}^{\des}_{321}(x,y,t)-1)) + \\
(xt+t)(\mathbb{WOP}^{\des}_{321}(x,y,t) -1 - xt 
\mathbb{WOP}^{\des}_{321}(x,y,t) -t(\mathbb{WOP}^{\des}_{321}(x,y,t)-1)).
\end{multline*}

\begin{figure}[ht]
	\centering
	\vspace{-1mm}
	\begin{tikzpicture}[scale =.5]
	\fill[black!20!white] (5,4) rectangle (9,9);
	\draw[help lines] (0,0) grid (9,9);
	\draw[ultra thick] (0,9) -- (0,6) -- (3,6)--(3,5)--(4,5)--(4,4)-- (5,4)--(5,1)--(7,1)--(7,0)--(9,0);
	\draw[help lines] (9,0)--(0,9);
	\fillll{7}{5}{(n-k,k)};
	%	\fillll{5}{0}{\textnormal{word: }314256897};
	\fillcross{1}{7};\fillcross{4}{6};\fillcross{2}{9};\fillcross{3}{8};\fillcross{5}{5};\fillcross{6}{2};\fillcross{7}{4};\fillcross{8}{1};\fillcross{9}{3};
	\end{tikzpicture}
	\caption{Ordered set partitions in Case 2 when the last return is at $(n-k,k)$.}
	\label{fig:LastRet}
\end{figure}
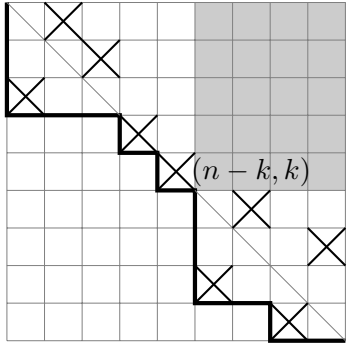

Now consider in the general case in Case 2 when the last return is at $(n-k,k)$ where $1< k \leq n-1$.  This situation is pictured in Figure 
\ref{fig:LastRet}. Because we 
fill the columns which do not have outer corner squares in 
a decreasing manner, reading from left to right, it is easy to see 
that there is no $\times$ in the squares of the shaded area in 
Figure \ref{fig:LastRet}.  This means that the $\times$s corresponding 
to $1,\ldots, k$ must be all in the bottom $k \times k$ squares. What we do not 
know is how the final increasing sequence of the elements 
$1, \ldots ,k$ in $w(\delta)$ union of the initial increasing sequence of 
the remaining elements break up into parts in $\delta$.  For example, 
in Figure \ref{fig:LastRet}, $k=4$ and the last increasing sequence 
of the elements $1, \ldots, 4$ in $w(\delta)$ is the single digit  $2$ and the initial 
increasing sequence of the remaining elements is $6,7,9,10$. 
Then we have two cases. The first case is when there is no overlap 
between the parts containing $1,\ldots,k$ and the remaining parts. 
In this case, we get a contribution of 
$NR(x,y,t)(\mathbb{WOP}^{\des}_{321}(x,y,t) -1)$ to 
$\mathbb{WOP}^{\des}_{321}(x,y,t)$.  If there is an overlap, then we 
need to remove the $x$ corresponding to the last part in the generating 
function $NR(x,y,t)$ so that we would get a contribution of
$\frac{1}{x}NR(x,y,t)(\mathbb{WOP}^{\des}_{321}(x,y,t) -1)$.

It follows that the total contribution to $\mathbb{WOP}^{\des}_{321}(x,y,t)$ 
from the ordered set partitions $\delta \in \mathcal{WOP}_n(321)$ in Case 2 is 
\begin{equation*}
NR(x,y,t) + (1+\frac{1}{x})NR(x,y,t)(\mathbb{WOP}^{\des}_{321}(x,y,t) -1).
\end{equation*}
Hence we have 
\begin{eqnarray*}
\mathbb{WOP}^{\des}_{321}(x,y,t) &=& 1 + xt\mathbb{WOP}^{\des}_{321}(x,y,t)
+ t(\mathbb{WOP}^{\des}_{321}(x,y,t)-1) + \nonumber\\
&&NR(x,y,t) + (1+\frac{1}{x})NR(x,y,t)(\mathbb{WOP}^{\des}_{321}(x,y,t) -1).
\end{eqnarray*}
This is a quadratic 
equation in $\mathbb{WOP}^{\des}_{321}(x,y,t)$ which we can solve 
to obtain the following theorem. \\
\ \\
\begin{theorem}The generating function
\begin{multline}\label{W321}
\mathbb{WOP}^{\des}_{321}(x,y,t) = \\
\frac{x+2 t (x+1)+2 t^2 (x+1) (x y-x-1)-x \sqrt{1-4 t (x+1) (t (x
		(y-1)-1)+1)}}{2 t (x+1)^2 (t (x
	(y-1)-1)+1)}.
\end{multline}
%\begin{multline}\label{W321}
%\mathbb{WOP}^{\des}_{321}(x,y,t) = \\
%\frac{2 t^2 (x+1) y^2 (x (y-1)+y)-2 t (x+1) y (x (y-1)+y)-x+x \sqrt{4 t (x+1) y (t
%		y-1)+1}}{2t(1+x)^2 y^2(ty-1)}.
%\end{multline}
\end{theorem}

Setting $y=1$ in (\ref{W321}), we obtain the following corollary which 
recovers the result of Chen, Dai and Zhou \cite{CDZ}.\\

\begin{corollary}The generating function
\begin{equation*}
\mathbb{WOP}_{321}(x,t) = \frac{x+2 t (1+x)-2 t^2 (1+x)-x \sqrt{1-4 (1-t) t
		(1+x)}}{2 (1-t) t (1+x)^2}.
\end{equation*}
\end{corollary}

The recursion that we used to compute 
$\mathbb{WOP}^{\des}_{321}(x,y,t)$ does not allow us to control the size of 
the parts of the ordered set partitions $\pi \in \WOP_n(321)$ so 
that we have not been able to compute generating functions 
of the form 
$\mathbb{WOP}^{\des}_{321,\{b_1, \ldots, b_k\}}(x,y,t,q_1, \ldots, q_k)$ 
in general.

\section{Generating functions for min-descents}
Based on the analysis in Section 2, we need to study the following 5 kinds of generating functions,
\begin{eqnarray*}
	\mathbb{WOP}^{\mindes}_{213}(x,y,t)\ &=&\mathbb{WOP}^{\mindes}_{312}(x,y,t),\\
	\mathbb{WOP}^{\mindes}_{132}(x,y,t),&&\mathbb{WOP}^{\mindes}_{231}(x,y,t),\\
	\mathbb{WOP}^{\mindes}_{123}(x,y,t),&&\mathbb{WOP}^{\mindes}_{321}(x,y,t).
\end{eqnarray*}
We are able to explicitly determine the functions $\mathbb{WOP}^{\mindes}_{132}(x,y,t)$, $\mathbb{WOP}^{\mindes}_{231}(x,y,t)$ and\\ $\mathbb{WOP}^{\mindes}_{213}(x,y,t) = \mathbb{WOP}^{\mindes}_{312}(x,y,t)$, and write the functions $\mathbb{WOP}^{\mindes}_{123}(x,y,t)$ and\\ $\mathbb{WOP}^{\mindes}_{321}(x,y,t)$ as roots of polynomial equations. 

\subsection{The function $\mathbb{WOP}^{\mindes}_{132}(x,y,t)$}

As we observed in Section 2, 
\begin{equation*}
\mathbb{WOP}^{\des}_{132}(x,y,t) = 
\mathbb{WOP}^{\mindes}_{132}(x,y,t),
\end{equation*}
thus we have the following theorem. 
\begin{theorem}The generating function
	\begin{multline*}
	\mathbb{WOP}^{\mindes}_{132}(x,y,t) =\mathbb{WOP}^{\des}_{132}(x,y,t)= \\
	\frac{(1+2yt+xyt-t-tx) -\sqrt{(1+2yt+xyt-t-tx)^2-4(1-t+ty)(t(y+xy))}}{2t(y+yx)},
	\end{multline*}
	and 
	\begin{equation*}
	\sum_{\pi \in \WOP_{n,k}(132)} y^{\mindes(\pi)} = 
	\frac{1}{k}\binom{n-1}{k-1}\sum_{j=0}^{k-1} 
	\binom{k}{j} \binom{n}{k-1-j} y^{k-1-j}.
	\end{equation*}
\end{theorem}

\subsection{The function $\mathbb{WOP}^{\mindes}_{231}(x,y,t)$}

Next consider $\mathbb{WOP}^{\mindes}_{231}(x,y,t)$. 
Let 
\begin{equation*}
C_n(x,y) := \sum_{\pi \in \WOP_n(231)}x^{\ell(\pi)}y^{\mindes(\pi)}.
\end{equation*}
We can classify ordered set partitions $\pi= B_1/\cdots/B_k \in \WOP_n(231)$ by the position $i$ 
of $n$ in the word of $\pi$.  Assume $n \geq 2$.\\
\ \\
{\bf Case 1.} $i=1$. \\
In this case $w(\pi)$ starts with $n$ which means that 
$n$ must be in a part by itself so that $B_1 =\{n\}$. 
Then $B_1$ contributes 
a factor of $xy$ since it automatically causes a min-descent 
with $B_2$. Thus the ordered set partitions 
$\pi \in \WOP_n(231)$ in Case 1 contribute 
$xyC_{n-1}(x,y)$ to $C_n(x,y)$.\\
\ \\
{\bf Case 2.} $i=n$. \\
In this case $w(\pi)$ ends with $n$. If $n$ is in a part by 
itself, then $B_k = \{n\}$ and there is no min-descent 
between $B_{k-1}$ and $B_k$. Hence we get a contribution of 
$xC_{n-1}(x,y)$ in this case. If $n \in B_{k}$   where $|B_k| \geq 2$,
then we can simply remove $n$ from $B_k$ and obtain an ordered 
set partition in $\WOP_n(231)$ with the same number of parts and 
the same number of min-descents, and we will get a contribution of 
$C_{n-1}(x,y)$. Thus the ordered set partitions 
$\pi \in \WOP_n(231)$ in Case 2 contribute 
$(1+x)C_{n-1}(x,y)$ to $C_n(x,y)$.\\
\ \\
{\bf Case 3.} $2 \leq i \leq n-1$.\\
In this case, $n$ must be the last element in some part $B_j$. 
Because $w(\pi)$ is 231-avoiding, it must be the case 
that all the elements in $B_1/\cdots/B_j-\{n\}$ are 
less than all the elements in $B_{j+1}/\cdots/B_k$. 
If $B_j= \{n\}$, then $B_j$ contributes a factor 
of $xy$ since $B_j$ will cause a min-descent with 
$B_{j+1}$. Our choices over all possibilities of 
$B_1/\cdots/B_{j-1}$ contribute a factor of $C_{i-1}(x,y)$ 
and our choices over all possibilities of 
$B_{j+1}/\cdots/B_{k}$ contribute a factor of $C_{n-i}(x,y)$. 
Thus we get a contribution of $xy C_{i-1}(x,y)C_{n-i}(x,y)$ 
in this case.  If $|B_{j}|\geq 2$, then we can eliminate 
$n$ from $B_j$.  Our choices over all possibilities of 
$B_1/\cdots/B_{j}-\{n\}$ contribute a factor of $C_{i-1}(x,y)$ 
and our choices over all possibilities of 
$B_{j+1}/\cdots/B_{k}$ contribute a factor of $C_{n-i}(x,y)$. 
Hence we get a contribution of $C_{i-1}(x,y)C_{n-i}(x,y)$ 
in this situation. Thus the ordered set partitions 
$\pi \in \WOP_n(231)$ in Case 3 contribute 
$(1+xy)C_{i-1}(x,y)C_{n-i}(x,y)$ to $C_n(x,y)$.\\
\ \\
It follows that for $n \geq 2$, 
\begin{equation*}
C_n(x,y) =(1+x+xy)C_{n-1}(x,y)+ \sum_{i=2}^{n-1}
(1+xy) C_{i-1}(x,y)C_{n-i}(x,y).
\end{equation*}

Hence,
\begin{eqnarray*}
	&& 	\mathbb{WOP}^{\mindes}_{231}(x,y,t)\nonumber\\
&=& 1+ xt+ 
	\sum_{n\geq 2} C_n(x,y)t^n \nonumber\\
	&=&1 +xt +(1+x+xy)t\sum_{n \geq 2} C_{n-1}(x,y)t^{n-1} + 
	(1+xy)t\sum_{n\geq 2} \sum_{k=2}^{n-1} C_{k-1}(x,y)C_{n-k}(x,y) \nonumber\\
	&=& 1+xt + (1+x+xy)t(\mathbb{WOP}^{\mindes}_{231}(x,y,t)-1)+ 
	(1+xy)t(\mathbb{WOP}^{\mindes}_{231}(x,y,t)-1)^2.
\end{eqnarray*}

This gives us a quadratic equation in 
which we can solve to prove the following theorem. 

\begin{theorem}The generating function
\begin{equation*}
\mathbb{WOP}^{\mindes}_{231}(x,y,t) =
\frac{1+t-tx+txy-\sqrt{(1+t-tx+txy)^2 -4(t+txy)}}{2(t+txy)}.
\end{equation*}
\end{theorem}

\subsection{The functions $\mathbb{WOP}^{\mindes}_{213}(x,y,t)=\mathbb{WOP}^{\mindes}_{312}(x,y,t)$}
As we observed in Section 2,
$$\mathbb{WOP}^\mindes_{213}(x,y,t) = \mathbb{WOP}^\maxdes_{132}(x,y,t).$$
Then we can work on the set $\WOP_n(132)$ and track the \maxdes\ statistic to compute the function $\mathbb{WOP}^\maxdes_{132}(x,y,t)$ instead of $\mathbb{WOP}^{\mindes}_{213}(x,y,t)$. 

We shall again classify the ordered set partitions 
$\pi \in \WOP_{n}(132)$ by the size of the last part and we will use 
the structure in Figure \ref{fig:132blocks}. 
Now suppose that 
$C(x,y,t) = \mathbb{WOP}^{\maxdes}_{132}(x,y,t)$.  In this 
case, we get a factor of $xt^r$ from the 
last part $\{a_1, \ldots, a_r\}$. Next 
we shall analyze when the last part from any 
$A_i$ will cause a max-descent in $\pi$. Let 
$s$ be the smallest index $i$ such that $A_i$ is non-empty. 
If $s =r+1$, then there is a max-descent from 
the last part of $A_{r+1}$ to $\{a_1, \ldots, a_r\}$ so that 
we would get a factor of $y(C(x,y,t)-1)$.  If 
$s \leq r$, then the last part of $A_s$ does not create 
a max-descent with $\{a_1, \ldots, a_r\}$ so it  contributes a factor of $(C(x,y,t)-1)$. However, 
each non-empty $A_j$ with $j > s$ creates
a max-descent between the last part of $A_j$ and the first 
part of the next non-empty $A_i$, so each such $A_j$ contributes a factor of 
$1+y(C(x,y,t)-1)$. Thus $C=C(x,y,t)$ satisfies the following 
recursive relation:
\begin{eqnarray}\label{132-blkrec}
C(x,y,t) &=& 1 + \sum_{r \geq 1} xt^r \left((1+y(C-1)+ 
\sum_{s=1}^r (C-1) (1+y(C-1))^{r+1-s}\right) \nonumber \\
&=& 1+ x(1+y(C-1))\sum_{r\geq 1}t^r  
\left(1+ (C-1)\sum_{s=1}^r (1+y(C-1))^{r-s}\right) \nonumber \\
&=& 1+ x(1+y(C-1))\sum_{r\geq 1}t^r  
\left(1+ (C-1)\frac{(1+y(C-1))^r-1}{(1+y(C-1))-1}\right) \nonumber \\
&=& 1+ x(1+y(C-1))\sum_{r\geq 1}t^r  
\left(1+ \frac{(1+y(C-1))^r-1}{y}\right)\nonumber \\
&=& 1+ x\frac{(1+y(C-1))}{y}\sum_{r\geq 1}t^r  
\left(y-1+(1+y(C-1))^r\right) \nonumber \\
&=&  1+ x\frac{(1+y(C-1))}{y}\left( \frac{t(y-1)}{1-t} + 
\frac{t(1+y(C-1))}{1-t(1+y(C-1))}\right)\nonumber \\
&=& 1+\frac{tx}{y}(1+y(C-1))\left(\frac{(y-1)}{1-t} + 
\frac{(1+y(C-1))}{1-t(1+y(C-1))}   \right).
\end{eqnarray}

Clearing the fractions gives a quadratic equation in $C$ which we can solve 
to show that 
\begin{equation*}%\label{C132-blkrec2} 
C(x,y,t) = 
\frac{P(x,y,t) -\sqrt{Q(x,y,t)}}{R(x,y,t)},
\end{equation*}
where 
\begin{eqnarray*}
	P(x,y,t) &=& 1-2t+t^2-tx +2ty-2t^2y+txy+2t^2xy-2t^2xy^2\\
	Q(x,y,t) &=& 1-4t+6t^2-4t^3+t^4 -2tx+4t^2x-2t^3x+t^2x^2-2txy+\\
	&& 4t^2xy-2t^3xy-2t^2x^2y+t^2x^2y^2, \ \mbox{and} \\
	R(x,y,t) &=& 2(ty-t^2y+txy-t^2xy^2).
\end{eqnarray*}

If we let $f(x,y,t) = C(x,y,t)-1$, then (\ref{132-blkrec}) gives 
that 
\begin{equation*}%\label{f-132} 
f(x,y,t) = \frac{tx}{y}(1+yf)\left( \frac{y-1}{1-t}+\frac{1+yf}{1-t(1+yf)}
\right) .
\end{equation*}
The Lagrange Inversion Theorem implies that 
the coefficient of $x^k$ in $f(x,y,t)$ is given by 
$$f(x,y,t)|_{x^k}= \frac{1}{k} \delta(x)^k\Big|_{x^{k-1}},$$
where 
$$\delta(x) =
\frac{t}{y}(1+yx)\left( \frac{y-1}{1-t}+\frac{1+yx}{1-t(1+yx)}
\right).$$ 
Thus,
\begin{eqnarray*}
	&&f(x,y,t)|_{x^kt^n} \nonumber\\
	&=&\frac{1}{k} 
	\frac{t^k}{y^k}(1+yx)^k\sum_{a=0}^k \binom{k}{a} 
	\frac{(y-1)^{k-a}}{(1-t)^{k-a}}\frac{(1+xy)^a}{(1-t(1+xy))^a}\bigg|_{x^{k-1}t^n} 
	 \nonumber\\
	&=&\frac{1}{k} 
	\frac{1}{y^k} \sum_{a=0}^k \binom{k}{a} 
	\frac{(y-1)^{k-a}}{(1-t)^{k-a}}\frac{(1+xy)^{k+a}}{(1-t(1+xy))^a}\bigg|_{x^{k-1}t^{n-k}}.
\end{eqnarray*}
By Newton's Binomial Theorem, we have 
\begin{eqnarray*}
	\frac{1}{(1-t)^{k-a}}&=& \sum_{u \geq 0} \binom{k-a+u-1}{u} t^u \ \mbox{ \ and} \\
	\frac{1}{(1-t(1+xy))^{a}} & =& \sum_{v\geq 0} \binom{a+v-1}{v} t^v(1+xy)^v. 
\end{eqnarray*}
It follows that 
\begin{eqnarray*}
	&&f(x,y,t)|_{x^k t^n} \nonumber\\
	&=&\frac{1}{k}\frac{1}{y^k} \sum_{a=0}^k \sum_{v=0}^{n-k}
	\binom{k}{a} \binom{a+v-1}{v} \binom{k-a+(n-k-v)-1}{n-k-v}(y-1)^{k-a}
	(1+xy)^{k+a+v}|_{x^{k-1}} \nonumber\\
	&=&\frac{1}{k}\frac{1}{y^k}\sum_{a=0}^k \sum_{v=0}^{n-k}
	\binom{k}{a}\binom{a+v-1}{v} \binom{k-a+(n-k-v)-1}{n-k-v}
	\binom{k+a+v}{k-1}(y-1)^{k-a} y^{k-1}\nonumber\\
	&=&\frac{1}{k y}\sum_{a=0}^k \sum_{v=0}^{n-k}
	\binom{k}{a}\binom{a+v-1}{v} \binom{k-a+(n-k-v)-1}{n-k-v}
	\binom{k+a+v}{k-1}(y-1)^{k-a}.
\end{eqnarray*}
Thus we have the following theorem.

\begin{theorem}The generating functions
	\begin{equation*}
\mathbb{WOP}^{\mindes}_{213}(x,y,t)=\mathbb{WOP}^{\mindes}_{312}(x,y,t)=	\mathbb{WOP}^{\maxdes}_{132}(x,y,t) = \frac{P(x,y,t) -\sqrt{(Q(x,y,t)}}{R(x,y,t)},
	\end{equation*}
	where 
	\begin{eqnarray*}
		P(x,y,t) &=& 1-2t+t^2-tx +2ty-2t^2y+txy+2t^2xy-2t^2xy^2,\\
		Q(x,y,t) &=& 1-4t+6t^2-4t^3+t^4 -2tx+4t^2x-2t^3x+t^2x^2-2txy+\nonumber\\
		&&4t^2xy-2t^3xy-2t^2x^2y+t^2x^2y^2, \ \mbox{and} \\
		R(x,y,t) &=& 2(ty-t^2y+txy-t^2xy^2),
	\end{eqnarray*}
	and the generating function
	\begin{multline*}
	\sum_{\pi \in \WOP_{n,k}(213)} y^{\mindes(\pi)} = \\
	\frac{1}{k y}\sum_{a=0}^k\sum_{v=0}^{n-k}
	\binom{k}{a}\binom{a+v-1}{v} \binom{k-a+(n-k-v)-1}{n-k-v}
	\binom{k+a+v}{k-1}(y-1)^{k-a}.
	\end{multline*}
\end{theorem}

We can compute the limit as $y$ approaches $0$ of 
$\mathbb{WOP}^{\mindes}_{213}(x,y,t)$ to obtain 
the generating function of ordered set partitions 
in $\WOP_{n}(213)$ which have no min-descents. 
In this case, we obtain the following 
corollary.
\begin{corollary}The generating function
\begin{equation*}
1 + \sum_{n \geq 1}t^n\sum_{\pi \in \WOP_{n}(213),\mindes(\pi) =0} x^{\ell(\pi)} = \frac{1+t(-2+t-tx)}{1+t^2-t(2+x)}.
\end{equation*}
\end{corollary}
Seting $x=1$, the coefficient list $\{a_n\}_{n\geq 0}$ in the Taylor series expansion is 
$$1,1,2,5,13,34,89,233,610,1597,4181,10946,28657, \ldots$$
which is a bisection of the Fibonacci numbers appears as sequence A001519 in the OEIS \cite{oeis} which has a large 
number of combinatorial interpretations. 
In fact, we can prove this combinatorially by showing the recurrence: $a_n=3a_{n-1}-a_{n-2}$ for all $n\geq 3$. Note that $a_n$ is the number of ordered set partitions that word-avoid 213 and have no min-descents. The number 1 must be in the first position in the word of each such ordered set partition. There are $a_{n-1}$ such ordered set partitions when 1 is in a block of size 1. When the number 1 is in a block of size larger than 1, we suppose that 2 is in the \thn{k} position in the word. It is easy to show that we have $a_{n-k+1}$ such ordered set partitions. Thus, 
\begin{equation*}
a_n=a_{n-1}+\sum_{k=2}^{n}a_{n-k+1} = a_{n-1}+\sum_{k=1}^{n-1}a_{k} =2a_{n-1}+\sum_{k=1}^{n-2}a_{k}=2a_{n-1}+(a_{n-1}-a_{n-2}),
\end{equation*}
which proves the recurrence relation.

Given any sequence of positive numbers $1 \leq b_1 < b_2 < \cdots < b_s$,
we let 
$$A=A(x,y,t,q_1, \ldots, q_s) = 
\mathbb{WOP}^{\mindes}_{213,\{b_1,\ldots,b_s\}}(x,y,t,q_1, \ldots, q_s).$$
It follows from the structure pictured in 
Figure \ref{fig:132blocks} and our analysis above that 
\begin{eqnarray*}
	A &=&  1 +\sum_{i=1}^s x q_i t^{b_i} ( 1+y(A-1)+ 
	\sum_{a=1}^{b_i} (A-1)(1+y(A-1))^{b_i+1-a}) \nonumber\\
	&=& 1 + \sum_{i=1}^s  x q_i t^{b_i} (1+y(A-1))\left( 1 + 
	\frac{(1+y(A-1))^{b_i} -1}{y}\right).
\end{eqnarray*}

If we set $F =F(x,y,t,q_1, \ldots,q_s) = 
A(x,y,t,q_1, \ldots, q_s)-1$, then 
we have 
\begin{equation*}
F = x \sum_{i=1}^s q_i t^{b_i} (1+yF)\left( 1 + 
\frac{(1+yF)^{b_i} -1}{y}\right).
\end{equation*}
It follows from the Lagrange Inversion Theorem that 
\begin{equation*}
F|_{x^k} = \frac{1}{k} \delta^k(x)|_{x^{k-1}}
\end{equation*}
where $\delta(x) = \sum_{i=1}^s q_i t^{b_i} (1+yx) \left( 1 + 
\frac{(1+yx)^{b_i} -1}{y}\right).$

One can use this expression to show that if 
$\alpha_1, \ldots, \alpha_s$ are non-negative 
integers such that 
$\sum_{i=1}^s\alpha_i  =k$ and $\sum_{i=1}^s \alpha_i b_i =n$, 
then 
\begin{equation*}
F|_{x^k t^n q_1^{\alpha_1} \cdots q_s^{\alpha_s}} = 
\frac{1}{k} \binom{k}{\alpha_1, \ldots, \alpha_s} 
\frac{(1+xy)^k}{y^k} \prod_{i=1}^s 
\left( (1+xy)^{b_i}-1\right)^{\alpha_i}\bigg|_{x^{k-1}}.
\end{equation*}
Hence it is possible to get a closed expression for 
$F|_{x^kt^nq_1^{\alpha_1} \cdots q_s^{\alpha_s}}$, and we shall 
omit the messy details.

\subsection{The function $\mathbb{WOP}^{\mindes}_{123,\{1,2\}}(x,y,t,q_1,q_2)$}

Next let us consider the computation 
of the generating function 
$$A(x,y,t,q_1,q_2) = 
\mathbb{WOP}^{\mindes}_{123,\{1,2\}}(x,y,t,q_1,q_2).$$  
We will again consider 
the case analysis of $\pi = 
B_1/ \cdots /B_j \in \WOP_{n,\{1,2\}}(123)$ by looking at
the first return of the path $P=\Psi(w(\pi))$ and we will keep the 
same notation. That is, we shall assume the 
first return is at $(n-k,k)$, $B_1/ \cdots /B_i$ are 
the parts containing the numbers $\{k+1, \ldots, n\}$ 
and $B_{i+1}/ \cdots /B_j$ are the parts containing 
the number $\{1, \ldots, k\}$. 

\ \\
{\bf Case 1.} The first return of $P$ is at the point 
$(1, n-1)$.\\
\ \\
In this case, we showed that  
$B_1 =\{n\}$.    
If $n=1$, then we get a contribution of 
$xtq_1$. Otherwise, $n$ will cause a min-descent 
between $B_1$ and $B_2$ 
which gives a contribution of $xtq_1y(A(x,y,t,q_1,q_2) -1)$. 
Thus, the contribution in this case is 
\begin{equation*}
xtq_1(1+y(A(x,y,t,q_1,q_2) -1)).
\end{equation*}
{\bf Case 2.} The first return of $P$ is at the point 
$(2,n-2)$.\\
\ \\
In this case, we showed that either  
$B_1 =\{n-1\}$ and $B_2=\{n\}$ or $B_1 =\{n-1,n\}$.  
It is easy to see that in the first case, the contribution to 
$A(x,y,t,q_1,q_2)$ is $x^2t^2q_1^2(1+y(A(x,y,t,q_1,q_2) -1))$. 
That is, if $n=2$, then we get a contribution of 
$x^2t^2q_1^2$. Otherwise, $B_2$ will cause a min-descent 
between $B_2$ and $B_3$ 
which gives a contribution of $x^2t^2q_1^2y(A(x,y,t,q_1,q_2) -1)$. 
Similarly, in the second case the contribution to $A(x,y,t,q_1,q_2)$ is 
$xt^2q_2(1+y(A(x,y,t,q_1,q_2)-1))$ as there is a
min-descent between $B_1$ and $B_2$ if $B_2$ exists. 
Thus the total contribution to
$A(x,y,t,q_1,q_2)$ from Case 2 is 
\begin{equation*}
(x^2t^2q_1^2 + xt^2q_2)(1+y(A(x,y,t,q_1,q_2)-1)).
\end{equation*}
\ \\
{\bf Case 3.} The first return of $P$ is at the point 
$(n-k,k)$ where $k<n-2$,  and $k+1$ is in column 
$n-k -1$.\\
In this case, we have the situation pictured in Figure 
\ref{fig:FirstRet2}. Thus $w(\pi) = w_1 \cdots w_n$ where 
$w_{n-k-1} = k+1$ and $w_{n-k} =p$ where $k+1 <p$. It follows 
that either $B_i =\{k+1,p\}$ or $B_{i-1}=\{k+1\}$ and $B_i=\{p\}$. 
We claim that the contribution to $A(x,y,t,q_1,q_2)$ in 
the first case where $B_i =\{k+1,p\}$ is 
\begin{equation*}
y(A(x,y,t,q_1,q_2)-1)xt^2q_2(1+y(A(x,y,t,q_1,q_2)-1)).
\end{equation*}
That is, the first factor $y$ comes from the fact 
that there is a min-descent between $B_{i-1}$ and 
$B_i$ since $\min(B_i)=k+1$ which 
is the smallest element in $B_1/\cdots /B_i$.
The next 
factor $(A(x,y,t,q_1,q_2)-1)$ comes from summing 
the weights of the reductions of $B_1/ \cdots /B_{i-1}$ over all 
possible choices of $B_1/\cdots/B_{i-1}$. The factor 
$xt^2q_2$ comes from $B_i$. If $B_{i+1}/ \cdots /B_j$ is empty 
then we get a factor of $1$, and if $B_{i+1}/ \cdots /B_j$ is not empty, 
then we get a factor of $y$, coming from the 
fact that the minimal element of $B_i$, $k+1$, is 
greater than the minimal element of $B_{i+1}$ which is 
some element in $\{1, \ldots, k\}$,  
and 
a factor of $(A(x,y,t,q_1,q_2)-1)$ comes from summing the weights 
over all possible choices of $B_{i+1}/ \cdots /B_j$.

A similar reasoning will show that the contribution to $A(x,y,t,q_1,q_2)$ in 
the second case where $B_{i-1} =\{k+1\}$ and $B_i= \{p\}$ is 
\begin{equation*}
y(A(x,y,t,q_1,q_2)-1)x^2t^2q_1^2(1+y(A(x,y,t,q_1,q_2)-1)).
\end{equation*}
Thus the total contribution to $A(x,y,t,q_1,q_2)$ in Case 3 is 
\begin{equation*}
y(A(x,y,t,q_1,q_2)-1)(xt^2q_2+x^2t^2q_1^2)(1+y(A(x,y,t,q_1,q_2)-1)).
\end{equation*}

At this point, our analysis differs from that of
$\mathbb{WOP}^{\des}_{123,\{1,2\}}(x,y,t,q_1,q_2)$.\\
\ \\
{\bf Case 4.} The first return of $P$ is at the point 
$(n-k,k)$ where $k<n-2$, $k+1$ is in column $r=n-k-2$, and $B_{i-1}$ has size 2. \\
Referring to \fref{FirstRet3},
in the word $w(\pi) = w_1 \cdots w_n$, we have 
$w_{n-k-2} = k+1$, $w_{n-k-1} =p_1$, and $w_{n-k} =p_2$, where $k+1 <p_2<p_1$. It follows 
that $B_i =\{p_2\}$, $B_{i-1}=\{k+1,p_1\}$, and there is no min-descent between $B_{i-1}$ and $B_i$. Referring to the Dyck path structure in \fref{last3} that if the path ends with $3$ right steps $RRR$ and it does not have a return, then there are two sub-Dyck-path components denoted $B$ in the picture -- the part tracking back from last step before the last down step to the step that it first reaches the first diagonal, and the part from the next step back to the start point. The corresponding parts of the two sub-Dyck-paths in the ordered set partition side are $B_1,\ldots,B_{i-2}$ that can be seen as $2$ ordered set partitions that word-avoid 123, whose contribution is $(1+y(A(x,y,t,q_1,q_2)-1))^2$. The contribution of parts $B_{i-1}$ and $B_i$ is $x^2t^3q_1q_2$ and the contribution of blocks $B_{i+1}/ \cdots /B_j$ is $(1+y(A(x,y,t,q_1,q_2)-1))$ for the same reason as Case 3. Thus the 
contribution of this case is 
\begin{equation*}
(1+y(A(x,y,t,q_1,q_2)-1))^2x^2t^3q_1q_2(1+y(A(x,y,t,q_1,q_2)-1)).
\end{equation*}
\begin{figure}[ht]
	\centering
	\vspace{-1mm}
	\begin{tikzpicture}[scale =.4]
	\draw (0,9)--(0,6)--(2,6)--(2,3)--(4,3)--(4,2)--(7,2)--(0,9)--(0,8)--(6,2)--(5,2)--(2,5);

	\node at (.5,6.5) {$B$};
	\node at (2.5,3.5) {$B$};
	\end{tikzpicture}
	\caption{The situation in Case 4.}
	\label{fig:last3}
\end{figure}
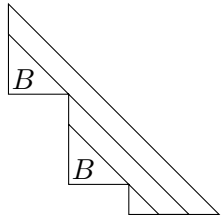

\noindent{\bf Case 5}. The first return of $P$ is at the point 
$(n-k,k)$ where $k < n-2$ and the size of $B_{i-1}$ is not 2 ($\pi$ does not satisfy Case 4). \\
This case is similar to Case 4 of $\mathbb{WOP}^{\des}_{123,\{1,2\}}(x,y,t,q_1,q_2)$ in Section 3.2. In this case, $B_i$ must be a singleton, and we claim that the contribution of this case is  
\begin{multline*}
y\left(A(x,y,t,q_1,q_2) - 1 -xtq_1(1+y(A(x,y,t,q_1,q_2) -1))-xt^2q_2(1+y(A(x,y,t,q_1,q_2) -1))^2\right)\\
\cdot xtq_1(1+y(A(x,y,t,q_1,q_2)-1)).
\end{multline*}
That is, the first factor $y$ comes from the fact 
that there is a min-descent caused by parts $B_{i-1}$ and $B_i$.  The next 
factor comes summing the weights of all possible choices 
of $B_1/ \cdots /B_{i-1}$. The contribution of part $B_i$ is $txq_1$ and the last factor $(1+y(A(x,y,t,q_1,q_2)-1))$ is the contribution of blocks $B_{i+1}/ \cdots /B_j$. 

Adding up the contributions leads to the following theorem.
\begin{theorem}The function $\mathbb{WOP}^{\mindes}_{123,\{1,2\}}(x,y,t,q_1,q_2)$ is the root of the following degree $3$ polynomial equation about $A$:
\begin{multline*}
A=1+txq_1(1+y(A-1))
	+(t^2xq_2+t^2x^2q_1^2)(1+y(A-1))^2
	+t^3x^2q_1q_2(1+y(A-1))^3\\
	\hspace*{1cm}+txyq_1(1+y(A-1))(A-1-txq_1(1+y(A-1))-t^2xq_2(1+y(A-1))^2).
\end{multline*}	
\end{theorem}

One can use Mathematica to compute the generating function:
%\begin{multline*}
%\mathbb{WOP}^{\mindes}_{123,\{1,2\}}(x,y,t,q_1,q_2)=1+q_1 t x+t^2 \left(q_1^2 x^2 y+q_1^2
%x^2+q_2 x\right)
%+
%t^3 \left(q_1^3 x^3 y^2+4 q_1^3 x^3 y+3
%q_1 q_2 x^2 y
%\right.\\\left.
%+q_1 q_2
%x\right)
%+t^4
%\left(q_1^4 x^4 y^3+11 q_1^4 x^4 y^2+2
%q_1^4 x^4 y+6 q_1^2 q_2 x^3 y^2+4
%q_1^2 q_2 x^3 y+5 q_1^2 q_2
%x^2 y+2 q_2^2 x^2 y\right)\\
%+t^5 \left(q_1^5 x^5
%y^4+26 q_1^5 x^5 y^3+15 q_1^5 x^5 y^2+10
%q_1^3 q_2 x^4 y^3+25 q_1^3
%q_2 x^4 y^2\right.\\\left.+16 q_1^3 q_2 x^3 y^2+5
%q_1^3 q_2 x^3 y+10 q_1 q_2^2
%x^3 y^2+5 q_1 q_2^2 x^2 y\right)+\cdots.
%\end{multline*}
\begin{multline*}
\mathbb{WOP}^{\mindes}_{123,\{1,2\}}(x,y,t,q_1,q_2)=1+ t xq_1+t^2 \left(q_2 x+q_1^2
x^2+q_1^2 x^2 y\right)+t^3 \left(q_1
q_2 x^2+3 q_1 q_2 x^2 y+4
q_1^3 x^3 y
\right.\\\left.
+q_1^3 x^3 y^2\right)+t^4
\left(2 q_2^2 x^2 y+9 q_1^2 q_2
x^3 y+2 q_1^4 x^4 y+6 q_1^2 q_2
x^3 y^2+11 q_1^4 x^4 y^2+q_1^4 x^4
y^3\right)\\
+t^5 \left(5 q_1 q_2^2 x^3 y+5
q_1^3 q_2 x^4 y+10 q_1 q_2^2
x^3 y^2+41 q_1^3 q_2 x^4 y^2
\right.\\\left.
+15
q_1^5 x^5 y^2+10 q_1^3 q_2 x^4
y^3+26 q_1^5 x^5 y^3+q_1^5 x^5
y^4\right)+\cdots.
\end{multline*}

\subsection{The function $\mathbb{WOP}^{\mindes}_{321}(x,y,t)$}

We write $C(x,y,t) = \mathbb{WOP}^{\mindes}_{321}(x,y,t)$. To study the function $C(x,y,t)$, we use the fact that the reverse of the word of any $\pi\in\WOP_n(321)$ is 123-avoiding. In other words, if we let $\overline{\WOP}_n(123)$ be the set of ordered set partitions whose numbers are organized in decreasing order inside each part and the word is 123-avoiding, then each  $\pi\in\WOP_n(321)$ corresponds to a $\bar{\pi}\in \overline{\WOP}_n(123)$. The \mindes\ of $\pi$ is then equal to the rise of the minimal elements of consecutive blocks (or \minrise) of $\bar{\pi}$.
We shall work on $\overline{\WOP}_n(123)$ and the statistic \minrise\ to compute the function $C(x,y,t)$.

We also need to define another generating function
\begin{equation*}
C_\ell (x,y,t):= 1 + \sum_{n \geq 1} t^n 
\sum_{\pi \in \WOPln} x^{\ell(\pi)} y^{|\{i: i<\ell(\pi)-1, B_i <_{min} B_{i+1}\}|}
\end{equation*}
that tracks the number of \minrise's that are not caused by the last two parts over all ordered set partitions in $\WOPln$.
	
We will always use the shorthand $C$ and $C_\ell$ for $C(x,y,t)$ and $C_\ell (x,y,t)$.

We start by studying the function $C(x,y,t)$. Note that the action {\em lift} defined in Section 3 preserves the \minrise\ of any ordered set partitions in $\WOPln$, which makes it possible to find a recursion for $\WOPln$ using the Dyck path bijection. For any $\pi=B_1/\cdots/B_m\in\WOPln$, we let $w(\pi)=w_1\cdots w_n\in S_n(123)$. Let the first return of the corresponding Dyck path be at the \thn{n-k} column and let $B_i$ be the part containing the number $w_{n-k}$. 

Then there are 5 cases.

\noindent{\bf Case 1.} $B_i$ has size $1$ and $w_{n-k-1}=k+1$.\\
In this case, there is a \minrise\ between parts $B_{i-1}$ and $B_i$.  The numbers before $k+1$ reduce to an ordered set partition in $\WOPlnn{n-k-2}$. Either $B_{i-1}$ only has the number $k+1$ or contains other numbers, and in the later case the \minrise\ caused by last two parts in the previous numbers is not counted. Thus the contribution of the numbers before  $w_{n-k}$ to  $C(x,y,t)$ is $tx(C+\frac{C_\ell -1}{x})$. Since the numbers after $w_{n-k}$ can form any ordered set partition in $\WOPlnn{k}$ and the \minrise\ is not affected, the contribution to the function $C(x,y,t)$ of this case is
\begin{equation*}
t^2x^2y \left(C+\frac{C_\ell-1}{x}\right) C.
\end{equation*}
{\bf Case 2.} $B_{i}$ has size larger than $1$ and $w_{n-k-1}=k+1$.\\
In this case, $B_{i}$ contains no number in $\{w_1,\ldots,w_{n-k-1}\}$ and there is no \minrise\ between parts $B_{i-1}$ and $B_i$. The contribution of the numbers before $w_{n-k}$ is  $tx(C+\frac{C_\ell -1}{x})$, and the contribution of the numbers from $w_{n-k}$ is $tx\left(\frac{C-1}{x}\right)$. The contribution to  $C(x,y,t)$ of this case is
\begin{equation*}
t^2x^2\left(C+\frac{C_\ell -1}{x}\right)\left(\frac{C-1}{x}\right).
\end{equation*}
{\bf Case 3.} $B_i$ has size $1$ and $w_{n-k-1}\neq k+1$.\\
In this case, there is no \minrise\ between parts $B_{i-1}$ and $B_i$. The contribution of the numbers before $w_{n-k}$ is  $\left(C-tx\left(C+\frac{C_\ell-1}{x}\right)\right)$. Since the numbers after $w_{n-k}$ form an ordered set partition in $\WOPlnn{k}$ and the first part can either contain the number  $w_{n-k}$ or not, without changing the \minrise, the contribution of the numbers from $w_{n-k}$ is $tx\left(C+\frac{C-1}{x}\right)$, and the contribution to the function $C(x,y,t)$ of this case is 
\begin{equation*}
tx\left(C+\frac{C-1}{x}\right)\left(C-tx\left(C+\frac{C_\ell-1}{x}\right)\right).
\end{equation*}
{\bf Case 4.} $w_{n-k-1}\in B_i$ and $w_{n-k+1}\notin B_i$.\\
In this case, there is no \minrise\ between parts $B_{i-1}$ and $B_i$. We have $w_{n-k}\neq k+1$ and $w_{n-k-1}\neq k+1$ in order to satisfy that $w_{n-k-1}\in B_i$. $w_{n-k+1}\notin B_i$ implies that the first part of the ordered set partition after $w_{n-k}$ does not contain the number  $w_{n-k}$. Thus the numbers up to $w_{n-k}$ contribute $t\left(C-1-tx\left(C+\frac{C_\ell-1}{x}\right)\right)$ and the numbers after $w_{n-k}$ contribute $C$ to the function $C(x,y,t)$. Thus the total contribution of this case is
\begin{equation*}
tC\left(C-1-tx\left(C+\frac{C_\ell-1}{x}\right)\right).
\end{equation*}
{\bf Case 5.} $w_{n-k-1}\in B_i$ and $w_{n-k+1}\in B_i$.\\
In this case, there is still no \minrise\ between parts $B_{i-1}$ and $B_i$. We have $w_{n-k}\neq k+1$ and $w_{n-k-1}\neq k+1$ in order to satisfy that $w_{n-k-1}\in B_i$. $w_{n-k+1}\in B_i$ implies that the first part of the ordered set partition after $w_{n-k}$ contains the number  $w_{n-k}$. As part $B_i$ connects the numbers before $w_{n-k}$ and the numbers after $w_{n-k}$, the \minrise\ caused by the last two parts before $w_{n-k}$ is not counted. Thus the numbers up to $w_{n-k}$ contribute $t\left(C_\ell-1-tx\left(C+\frac{C_\ell-1}{x}\right)\right)$ and the numbers after $w_{n-k}$ contribute $\frac{C-1}{x}$ to the function $C(x,y,t)$. The total contribution of this case is
\begin{equation*}
t\left(\frac{C-1}{x}\right)\left(C_\ell-1-tx\left(C+\frac{C_\ell-1}{x}\right)\right).
\end{equation*}
Summing the contribution of all the five cases, we have
\begin{multline}\label{thm20eq1}
C(x,y,t)=1+(y-1)t^2x^2C\left(C+\frac{C_\ell-1}{x}\right) +txC\left(C+\frac{C-1}{x}\right)\\
+tC\left(C-1-tx\left(C+\frac{C_\ell-1}{x}\right)\right) + t\left(\frac{C-1}{x}\right)\left(C_\ell-1-tx\left(C+\frac{C_\ell-1}{x}\right)\right).
\end{multline}

We can do similar analysis for $C_\ell(x,y,t)$. We have the following $7$ cases, of which the first $5$ cases are similar to that of $C(x,y,t)$.

\noindent{\bf Case 1.} $B_i$ has size $1$, $w_{n-k-1}=k+1$ and $k>0$.\\
The argument is same as Case 1 of $C(x,y,t)$ except that the contribution of the numbers after $w_{n-k}$ is $C_\ell-1$ instead of $C$, since $k>0$ implies that $B_{i+1}$ is not empty, and we do not count the \minrise\ between the last two parts of $\pi$. Thus the contribution to $C_\ell(x,y,t)$ of this case is
\begin{equation*}
t^2x^2y \left(C+\frac{C_\ell-1}{x}\right) (C_\ell-1).
\end{equation*}
{\bf Case 2.} $B_{i}$ has size larger than $1$ and $w_{n-k-1}=k+1$.\\
Similar to Case 2 of $C(x,y,t)$, the contribution is
$
t^2x^2\left(C+\frac{C_\ell -1}{x}\right)\left(\frac{C_\ell-1}{x}\right).
$
The only difference is that the contribution of numbers after $w_{n-k}$ is $\frac{C_\ell-1}{x}$ instead of $\frac{C-1}{x}$ as we do not count the \minrise\ between the last two parts.

\noindent{\bf Case 3.} $B_i$ has size $1$, $w_{n-k-1}\neq k+1$ and $k>0$.\\
Similar to Case 3 of $C(x,y,t)$, the contribution is $tx\left(C_\ell-1+\frac{C_\ell-1}{x}\right)\left(C-tx\left(C+\frac{C_\ell-1}{x}\right)\right).$ 
The difference is that the contribution of numbers after $w_{n-k}$ is $\left(C_\ell-1+\frac{C_\ell-1}{x}\right)$ as we do not count the \minrise\ between the last two parts and the collection of numbers after $w_{n-k}$ is not empty.

\noindent{\bf Case 4.} $w_{n-k-1}\in B_i$, $w_{n-k+1}\notin B_i$ and $k>0$.\\
Similar to Case 4 of $C(x,y,t)$, the contribution is
$
t(C_\ell-1)\left(C-1-tx\left(C+\frac{C_\ell-1}{x}\right)\right).
$
The contribution of numbers after $w_{n-k}$ is $(C_\ell-1)$ since $k>0$ implies that the collection of numbers after $w_{n-k}$ is not empty.

\noindent{\bf Case 5.} $w_{n-k-1}\in B_i$ and $w_{n-k+1}\in B_i$.\\
Similar to Case 5 of $C(x,y,t)$, the contribution is
$
t\left(\frac{C_\ell-1}{x}\right)\left(C_\ell-1-tx\left(C+\frac{C_\ell-1}{x}\right)\right).
$
The the contribution of numbers after $w_{n-k}$ is $\frac{C_\ell-1}{x}$ as we do not count the \minrise\ between the last two parts.

\noindent{\bf Case 6.} $k=0$ and $w_{n-k-1}\notin B_i$.\\
In this case, $B_i=\{w_{n-k}\}$. Since we do not count the descents of the last two parts, we do not care whether $w_{n-k}$ is bigger or smaller than the minimum of the previous part. The contribution of this case is $txC$.

\noindent{\bf Case 7.} $k=0$ and $w_{n-k-1}\in B_i$.\\
In this case, $B_i$ can be seen as including $w_{n-k}$ in the last part before $w_{n-k}$. The last \minrise\ before $w_{n-k}$ is not counted, and $w_{n-k},w_{n-k-1}\neq k+1$. The contribution of this case is $t\left(C_\ell-1-tx\left(C+\frac{C_\ell-1}{x}\right)\right).$

Summing the contribution of all the 7 cases, we have
\begin{multline}\label{thm20eq2}
	C_\ell(x,y,t)=1+(y-1)t^2x^2(C_\ell-1)\left(C+\frac{C_\ell-1}{x}\right) +txC\left(C_\ell-1+\frac{C_\ell-1}{x}\right)+txC\\
	+t\left(C_\ell-1-tx\left(C+\frac{C_\ell-1}{x}\right)\right)+tx\left(C_\ell-1+\frac{C_\ell-1}{x}\right)\left(C-tx\left(C+\frac{C_\ell-1}{x}\right)\right)\\
	+
	t(C_\ell-1)\left(C-1-tx\left(C+\frac{C_\ell-1}{x}\right)\right).
\end{multline}

Using equations (\ref{thm20eq1}) and (\ref{thm20eq2}) about $C(x,y,t)$ and $C_\ell (x,y,t)$, we can compute the Groebner basis of the functions  to find an equation that $C(x,y,t)$ satisfies, and we have the following theorem.

\begin{theorem}The function $\mathbb{WOP}^{\mindes}_{321}(x,y,t)$ is the root of the following degree $4$ polynomial equation about $C$:
%\begin{multline*}
%1+t(-1+t-t^2+x (2+2 x-x y))+C \left(-2+t (-3+t (3+x (-4+3 x (-2+y)))+x (-5+x (-3+y))\right.\\
%\left.+t^2 (2+x (3+3 x-2 x y)))\right)
%+C^2 \left(1+t ((3+x)^2-t^3 (-3+x^2)+t (3+x) (-1+2 x (1+x))\right.\\
%\left.-t^2 (10+x (6+x (3+x (4+x (-2+y) (-1+y)-3 y)-y))))\right)
%+C^3 t \left(-5-3 x+t
%(-7-x (8\right.\\
%\left.+x (3+x) (1+y))+t^2 (-6+x (-6+x (-3+5 y+x (1+x+y-x y))))+t (18+x (17\right.\\
%\left.+x (6-6 y+x (2+x-(4+x) y)))))\right)
%+C^4 t^2 (2+x-t (1+x+x^2)+t x^2 y)\\
% \left(3+2 x+t (-3-x (3+x)+x^2 (2+x)
%y)\right)=0.
%\end{multline*}
\begin{multline*}
1+t \left(-1+2 x+2 x^2-x^2 y\right)+t^2-t^3+C \left(-2+t \left(-3-5 x-3 x^2+x^2 y\right)+t^2 \left(3-4 x-6 x^2+3 x^2
y\right)
\right.\\
\left.
+t^3 \left(2+3 x+3 x^2-2 x^2 y\right)\right)
+C^2 \left(1+t \left(9+6 x+x^2\right)+t^2 \left(-3+5 x+8 x^2+2 x^3\right)\right.\\
\left.+t^3
\left(-10-6 x
-3 x^2-4 x^3-2 x^4+x^2 y+3 x^3 y+3 x^4 y-x^4
y^2\right)+t^4 \left(3-x^2\right)\right)\\
+C^3 \left(t (-5-3 x)+t^2 \left(-7-8 x-3 x^2-x^3-3 x^2 y-x^3 y\right)+t^3 \left(18+17
x+6 x^2+2 x^3
\right.\right.\\
\left.\left.
+x^4-6 x^2 y-4 x^3 y-x^4 y\right)+t^4 \left(-6-6 x-3
x^2+x^3+x^4+5 x^2 y+x^3 y-x^4 y\right)\right)\\
+C^4 t^2 \left(2-t+x-t x-t x^2+t x^2 y\right) \left(3-3 t+2 x-3 t x-t x^2+2 t
x^2 y+t x^3 y\right)=0.
\end{multline*}
\end{theorem}
One can use Mathematica to compute the generating function:
\begin{multline*}
\mathbb{WOP}^{\mindes}_{321}(x,y,t)=1+t x+t^2 \left(x^2
y+x^2+x\right)
+t^3 \left(4 x^3
y+x^3+2 x^2 y+5 x^2+2 x\right)\\
+t^4 \left(2 x^4 y^2+11 x^4 y+x^4+17 x^3
y+17 x^3+4 x^2 y+22 x^2+6 x\right)
+t^5 \left(15 x^5 y^2+26 x^5 y\right.\\
\left.+x^5+10 x^4 y^2+90 x^4
y+49 x^4+65 x^3 y+123 x^3+10 x^2 y+88 x^2+18
x\right)+\cdots.
\end{multline*}

\section{Generating functions for part-descents}

In this section, we shall study the generating function 
$\mathbb{WOP}^{\pdes}_{\alpha}(x,y,t)$ where
$\alpha \in S_3$.  Based on the analysis in Section 2, we need to study the following 4 kinds of generating functions,
\begin{eqnarray*}
\mathbb{WOP}^{\pdes}_{132}(x,y,t)&=&\mathbb{WOP}^{\pdes}_{213}(x,y,t),\\
\mathbb{WOP}^{\pdes}_{231}(x,y,t)&=&\mathbb{WOP}^{\pdes}_{312}(x,y,t),\\
\mathbb{WOP}^{\pdes}_{123}(x,y,t),&&\mathbb{WOP}^{\pdes}_{321}(x,y,t).
\end{eqnarray*}
We are able to explicitly determine the functions $\mathbb{WOP}^{\pdes}_{132}(x,y,t)=\mathbb{WOP}^{\pdes}_{213}(x,y,t)$, and write the functions $\mathbb{WOP}^{\pdes}_{231}(x,y,t)=\mathbb{WOP}^{\pdes}_{312}(x,y,t)$ and $\mathbb{WOP}^{\pdes}_{321}(x,y,t)$ as roots of polynomial equations. We fail to obtain a recursive formula for $\mathbb{WOP}^{\pdes}_{123}(x,y,t)$ since it is hard to get the \pdes\ statistic under the the \textit{lift} action of a 123-avoiding permutation.

\subsection{The functions $\mathbb{WOP}^{\pdes}_{132}(x,y,t)=\mathbb{WOP}^{\pdes}_{213}(x,y,t)$}

As we observed in Section 2, 
$$\mathbb{WOP}^{\pdes}_{132}(x,y,t) = 
\mathbb{WOP}^{\mindes}_{213}(x,y,t).$$
Thus we have the following theorem. 

\begin{theorem}The generating functions
	\begin{eqnarray*}
	\mathbb{WOP}^{\pdes}_{132}(x,y,t)&=&\mathbb{WOP}^{\pdes}_{213}(x,y,t)\nonumber\\
	&=&	\mathbb{WOP}^{\maxdes}_{132}(x,y,t)=\mathbb{WOP}^{\mindes}_{213}(x,y,t)=\mathbb{WOP}^{\mindes}_{312}(x,y,t)\nonumber\\
	& = &\frac{P(x,y,t) -\sqrt{Q(x,y,t)}}{R(x,y,t)},
	\end{eqnarray*}
	where 
	\begin{eqnarray*}
		P(x,y,t) &=& 1-2t+t^2-tx +2ty-2t^2y+txy+2t^2xy-2t^2xy^2,\\
		Q(x,y,t) &=& 1-4t+6t^2-4t^3+t^4 -2tx+4t^2x-2t^3x+t^2x^2-2txy+\nonumber\\
		&&4t^2xy-2t^3xy-2t^2x^2y+t^2x^2y^2, \ \mbox{and} \\
		R(x,y,t) &=& 2(ty-t^2y+txy-t^2xy^2),
	\end{eqnarray*}
	and 
	\begin{equation*}
	\sum_{\pi \in \WOP_{n,k}(132)} \hskip -5mm y^{\pdes(\pi)} = 
	\frac{1}{k}\frac{1}{y}\sum_{a=0}^k\sum_{v=0}^{n-k}
	\binom{k}{a}\binom{a+v-1}{v} \binom{k-a+(n-k-v)-1}{n-k-v}
	\binom{k+a+v}{k-1}(y-1)^{k-a}.
	\end{equation*}
\end{theorem}

\subsection{The functions $\mathbb{WOP}^{\pdes}_{231}(x,y,t)=\mathbb{WOP}^{\pdes}_{312}(x,y,t)$}

We compute the function $\mathbb{WOP}^{\pdes}_{312}(x,y,t)$ and write $D(x,y,t) = \mathbb{WOP}^{\pdes}_{312}(x,y,t)$. As this is different from the $132$-avoiding case, we will consider a new structure for the set $\WOP_n(312)$.

Given any ordered set partition $\pi= B_1/\cdots/B_k \in \WOP_n(312)$. If the size $n=0$, then it contributes $1$ to the function $D(x,y,t)$. Otherwise, $\pi$ has at least one part and we suppose the last part is $B_k=\{a_1, a_2,\ldots,a_r\}$ with $r\geq1$ numbers. Note that there is no number $a>a_2$ in the previous blocks ${B_1,\ldots,B_{k-1}}$, otherwise the subsequence $(a,a_1,a_2)$ of $w(\pi)$ is a $312$-occurrence. Thus, the subsequence $a_2,\ldots,a_r$ must be a consecutive integer sequence.

Now, we divide the numbers in the previous blocks ${B_1,\ldots,B_{k-1}}$ into $2$ sets: let $A_1=\{1,\ldots,a_1-1\}$ be the numbers smaller than $a_1$ and $A_2=\{a_1+1,\ldots,a_2-1\}$ be the numbers bigger than $a_1$. The numbers in the set $A_1$ must appear before the numbers in $A_2$ as otherwise there is a $312$-occurrence in the word. Thus, an ordered set partition $\pi= B_1/\cdots/B_k \in \WOP_n(312)$ has the structure pictured in \fref{3}.

\begin{figure}[ht]
	\centering
	\vspace{-1mm}
	\begin{tikzpicture}[scale =.5]
	\draw[thick] (0,0) rectangle (2,2);
	\draw[thick] (2,2) rectangle (4,4);
%	\draw[thick] (2,2) circle [radius=0.4];
	\draw[fill] (5.3,5.3) circle [radius=0.05];
	\draw[fill] (5,5) circle [radius=0.05];
	\draw[fill] (5.6,5.6) circle [radius=0.05];
	\filll{1}{1}{A_1};
	\filll{3}{3}{A_2};
	\filll{4.5}{4.5}{\emptyset};
	\filll{6.5}{6.5}{\emptyset};
	\filll{7.5}{7.5}{\emptyset};
	\draw (9,-2) -- (9,8);
	\draw (4,2) -- (9.3,2) node[align=left, right] {$a_1$};
	\draw (4,4) -- (9.6,4) node[align=left, right] {$a_2$};
	\draw (6,6) -- (9.9,6) node[align=left, right] {$a_{r-1}$};
	\draw (7,7) -- (10.1,7) node[align=left, right] {$a_r$};
	\draw[fill] (10.3,4.5) circle [radius=0.05];
	\draw[fill] (10.4,5) circle [radius=0.05];
	\draw[fill] (10.5,5.5) circle [radius=0.05];
	\filll{2}{-1}{B_1\cdots B_{k-1}};\filll{10}{-1}{B_{k}};
	\end{tikzpicture}
	\vspace*{-3mm}
	\caption{Structure of an ordered set partition in $\WOP_n(312)$.}
	\label{fig:3}
\end{figure}
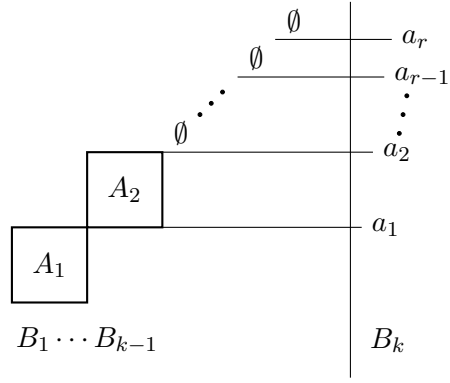

We let $A_i(\pi)$ be the restriction of $\pi$ to the set $A_i$. Then each $A_i(\pi)$ is also an ordered set partition in $\WOP_n(312)$. However, if both $A_i$'s are not empty, then it is possible that the last block of $A_1$ and the first block of $A_2$ are contained in the same block in $\pi$. In that case, the \pdes\ caused by the last two blocks of $A_1$ (if any) and the \pdes\ caused by the first two blocks in $A_2$ (if any) will not contribute to $\pdes(\pi)$. We let $D_\ell (x,y,t)$, $D_f (x,y,t)$ and $D_{\ell f}(x,y,t)$ be the generating functions tracking the number of \pdes\ without tracking the \pdes\ caused by the last two parts, the first two parts, and both last and first two parts that
\begin{eqnarray*}
D_\ell (x,y,t)&:=& 1 + \sum_{n \geq 1} t^n 
\sum_{\pi \in \WOP_n(312)} x^{\ell(\pi)} y^{|\{i: i<\ell(\pi)-1, B_i >_{p} B_{i+1}\}|}, \\
D_f (x,y,t)&:=& 1 + \sum_{n \geq 1} t^n 
\sum_{\pi \in \WOP_n(312)} x^{\ell(\pi)} y^{|\{i: i>1, B_i >_{p} B_{i+1}\}|}, \\
D_{\ell f} (x,y,t)&:=&1 + \sum_{n \geq 1} t^n 
\sum_{\pi \in \WOP_n(312)} x^{\ell(\pi)} y^{|\{i: 1<i<\ell(\pi)-1, B_i >_{p} B_{i+1}\}|},
\end{eqnarray*}
then we can compute the recursive equations of functions $D(x,y,t)$, $D_\ell (x,y,t)$, $D_f (x,y,t)$ and $D_{\ell f}(x,y,t)$ respectively.

We first consider the function $D(x,y,t)$. 

\noindent{\bf Case 1.} The last part $B_k$ has size  bigger than $1$. \\
Then there is always no \pdes\ involving the part $B_k$ as the last part contains the number $a_2$ which is greater than any numbers in $B_1,\ldots,B_{k-1}$. The last part has contribution $tx^2+tx^3+\cdots=\frac{tx^2}{1-t}$, and the contribution of $B_1,\ldots,B_{k-1}$ is $D^2(x,y,t)$ when the last block of $A_1$ and the first block of $A_2$ are in different blocks in $\pi$, and $\frac{(D_\ell (x,y,t)-1)(D_f (x,y,t)-1)}{x}$ when the last block of $A_1$ and the first block of $A_2$ are in the same block in $\pi$. Thus, the contribution of this case to the function $D(x,y,t)$ is
\begin{equation*}
\frac{tx^2}{1-t}\left( D^2(x,y,t)+\frac{(D_\ell (x,y,t)-1)(D_f (x,y,t)-1)}{x}  \right).
\end{equation*}
{\bf Case 2.} $B_k$ has size $1$, $A_2$ only contains $1$ block which is in the same block as the last block of $A_1$ in $\pi$.\\
In this case, the set $A_1$ cannot be empty and there is still no \pdes\ caused by the last two parts of $\pi$. The contribution is
\begin{equation*}
tx\left((D_\ell (x,y,t)-1)\frac{t}{1-t}\right).
\end{equation*}
{\bf Case 3.} $B_k$ has size $1$, $A_2$ is empty.\\
In this case, there is no \pdes\ caused by the last two parts of $\pi$ and the contribution is
\begin{equation*}
tx D(x,y,t).
\end{equation*}
{\bf Case 4.} $B_k$ has size $1$, and $\pi$ does not satisfy Case 2 or 3.\\
In this case, there is a \pdes\ caused by the last two parts of $\pi$. Since it is possible that the last block of $A_1$ and the first block of $A_2$ are in the same block in $\pi$, the contribution of this case is
\begin{equation*}
txy\left(D(x,y,t)(D(x,y,t)-1)+\frac{(D_\ell (x,y,t)-1)(D_f (x,y,t)-\frac{tx}{1-t}-1)}{x}\right).
\end{equation*}
Summing the contribution of all the 4 cases, and we write $D,D_\ell,D_f ,D_{\ell f}$ on the right hand side to abbreviate $D(x,y,t),D_\ell (x,y,t),D_f (x,y,t),D_{\ell f}(x,y,t)$, then
we have
\begin{equation}\label{t22eq1}
D(x,y,t)=1+
\frac{tx}{1-t}\left( D^2+\frac{(D_\ell -1)(D_f -1)}{x}  \right)\\
+(y-1)tx\left(D(D-1)+\frac{(D_\ell -1)(D_f -\frac{tx}{1-t}-1)}{x}\right).
\end{equation}

For the function $D_\ell (x,y,t)$, we do not need to consider the contribution to part-descent involving part $B_k$, thus the analysis is like Case 1 of  $D(x,y,t)$ and we have
\begin{equation}\label{t22eq2}
D_\ell(x,y,t)=1+
\frac{tx}{1-t}\left( D^2+\frac{(D_\ell -1)(D_f -1)}{x}  \right).
\end{equation}

For the function $D_f (x,y,t)$, we have similar cases to  $D (x,y,t)$, but one more case when last part is of size $1$.

\noindent{\bf Case 1.} $B_k$ has size larger than $1$. \\
In this case, there is always  no \pdes\ involving part $B_k$. The last part has contribution $\frac{tx^2}{1-t}$. The contribution of $B_1,\ldots,B_{k-1}$ is $(D_f(x,y,t)-1)D(x,y,t)$ when $A_1$ is not empty and the last block of $A_1$ and  the first block of $A_2$ are in different blocks in $\pi$,  $D_f(x,y,t)$ when $A_1$ is empty, and $\frac{(D_{\ell f} (x,y,t)-1)(D_f (x,y,t)-1)}{x}$ when the last block of $A_1$ and the first block of $A_2$ are in the same block in $\pi$. Thus, the contribution of this case to the function $D(x,y,t)$ is
\begin{equation*}
\frac{tx^2}{1-t}\left( (D_f(x,y,t)-1)D(x,y,t)+D_f(x,y,t)+\frac{(D_{\ell f} (x,y,t)-1)(D_f (x,y,t)-1)}{x}  \right).
\end{equation*}
{\bf Case 2.} $B_k$ has size $1$, $A_2$ only contains $1$ block and it is in the same block as the last block of $A_1$.\\
In this case, the set $A_1$ cannot be empty and there is still no \pdes\ caused by the last two parts of $\pi$. The contribution is
\begin{equation*}
tx\left((D_{\ell f} (x,y,t)-1)\frac{t}{1-t}\right).
\end{equation*}
{\bf Case 3.} $B_k$ has size $1$, $A_2$ is empty.\\
In this case, there is no \pdes\ caused by the last two parts of $\pi$ and the contribution is
\begin{equation*}
tx D_f(x,y,t).
\end{equation*}
{\bf Case 4.} $B_k$ has size $1$, $A_1$ is empty, and $A_2$ only has one block.\\
In this case,  the \pdes\ caused by the only two parts of $\pi$ is not counted as we do not count the first \pdes, and the contribution is
\begin{equation*}
tx \frac{tx}{1-t}.
\end{equation*}
\noindent{\bf Case 5.} $B_k$ has size $1$, and the numbers in sets $A_1,A_2$ does not satisfy Case 2, 3 or 4.\\
In this case, there is a \pdes\ caused by the last two parts of $\pi$. Since it is possible that the last block of $A_1$ and the first block of $A_2$ are in the same block, the contribution of this case is
\begin{multline*}
txy\left((D_f(x,y,t)-1)(D(x,y,t)-1)+(D_f(x,y,t)-1-\frac{tx}{1-t})\right.\\\left.+\frac{(D_{\ell f} (x,y,t)-1)(D_f (x,y,t)-\frac{tx}{1-t}-1)}{x}\right).
\end{multline*}

Summing the contribution of all the 5 cases, we have
\begin{multline}\label{t22eq3}
D_f(x,y,t)=	1+\frac{tx}{1-t}\left( (D_f-1)D+D_f+\frac{(D_{\ell f} -1)(D_f -1)}{x}  \right)\\+(y-1)tx\left((D_f-1)D-\frac{tx}{1-t}+\frac{(D_{\ell f} -1)(D_f -\frac{tx}{1-t}-1)}{x}\right).
\end{multline}

For the function $D_{\ell f}(x,y,t)$, we do not need to consider the contribution to part-descent involving part $B_k$, thus the contribution is like Case 1 of  $D_f(x,y,t)$ and we have
\begin{equation}\label{t22eq4}
D_{\ell f}(x,y,t)=1+
\frac{tx}{1-t}\left((D_f-1)D+D_f+\frac{(D_{\ell f} -1)(D_f -1)}{x}  \right).
\end{equation}

Using equations (\ref{t22eq1}), (\ref{t22eq2}), (\ref{t22eq3}) and (\ref{t22eq4}) about $D(x,y,t)$, $D_\ell (x,y,t)$, $D_f (x,y,t)$ and $D_{\ell f}(x,y,t)$, we can compute the Groebner basis of the functions  to find an equation that $D(x,y,t)$ satisfies, and we have the following theorem.

\begin{theorem}We have\\
$
D(x,y,t)=1+
\frac{tx}{1-t}\left( D^2+\frac{(D_\ell -1)(D_f -1)}{x}  \right)
+(y-1)tx\left(D(D-1)+\frac{(D_\ell -1)(D_f -\frac{tx}{1-t}-1)}{x}\right),\\
$
$
D_\ell(x,y,t)=1+
\frac{tx}{1-t}\left( D^2+\frac{(D_\ell -1)(D_f -1)}{x}  \right),\\
$
\resizebox{\textwidth}{!}{$D_f(x,y,t)=	1+\frac{tx}{1-t}\left( (D_f-1)D+D_f+\frac{(D_{\ell f} -1)(D_f -1)}{x}  \right)+(y-1)tx\left((D_f-1)D-\frac{tx}{1-t}+\frac{(D_{\ell f} -1)(D_f -\frac{tx}{1-t}-1)}{x}\right),\\$}
$
D_{\ell f}(x,y,t)=1+
\frac{tx}{1-t}\left((D_f-1)D+D_f+\frac{(D_{\ell f} -1)(D_f -1)}{x}  \right),\\
$
and the function $\mathbb{WOP}^{\pdes}_{312}(x,y,t)$ is the root of the following degree $3$ polynomial equation about $D$:
\begin{multline*}
1-t+D (-1+t) (1+t (1+2 x (-1+y)))+D^2 (1-t) t \left(1+t x^2
(-1+y)^2+x (-1+t (-1+y)+2 y)\right)\\
+D^3 t^2 x (-1+y)
(-1+t (1+x (-1+y))-x y)=0.
\end{multline*}
\end{theorem}
We can use Mathematica to compute the generating function:
\begin{multline*}
\mathbb{WOP}^{\pdes}_{312}(x,y,t)=1+t x+t^2 \left(x^2 y+x^2+x\right)+t^3 \left(x^3 y^2+3 x^3 y+x^3+x^2
y+4 x^2+x\right)
+t^4 \left(x^4
y^3\right.\\\left.+6 x^4 y^2+6 x^4 y+x^4+x^3 y^2+11 x^3 y+9 x^3+x^2
y+8 x^2+x\right)
+
t^5 \left(x^5 y^4+10 x^5 y^3+20 x^5 y^2\right.\\\left.+10 x^5
y+x^5+x^4 y^3+21 x^4 y^2+46 x^4 y+16 x^4+x^3 y^2+23
x^3 y+32 x^3+x^2 y+13 x^2+x\right)+\cdots.
\end{multline*}

\subsection{The function $\mathbb{WOP}^{\pdes}_{321}(x,y,t)$}
We write $D(x,y,t) = \mathbb{WOP}^{\pdes}_{321}(x,y,t)$. As we defined in Section 4.5, $\overline{\WOP}_n(123)$ is the set of ordered set partitions whose numbers are organized in decreasing order inside each part and the word is 123-avoiding. Each  $\pi\in\WOP_n(321)$ corresponds to a $\bar{\pi}\in \overline{\WOP}_n(123)$, and the \pdes\ of $\pi$ is equal to the part-rise (or \prise) of $\bar{\pi}$.
We want to work on $\overline{\WOP}_n(123)$ and the statistic \prise\ to compute the function $D(x,y,t)$.

We also need to define $D_\ell (x,y,t)$, $D_f (x,y,t)$ and $D_{\ell f}(x,y,t)$ as the generating functions tracking the number of \prise\ without tracking the \prise\ caused by the last two parts, the first two parts, and both last and first two parts of ordered set partitions in $\WOPln$ that
\begin{eqnarray*}
	D_\ell (x,y,t)&:=& 1 + \sum_{n \geq 1} t^n 
	\sum_{\pi \in \WOPln} x^{\ell(\pi)} y^{|\{i: i<\ell(\pi)-1, B_i <_p B_{i+1}\}|}, \\
	D_f (x,y,t)&:=& 1 + \sum_{n \geq 1} t^n 
	\sum_{\pi \in \WOPln} x^{\ell(\pi)} y^{|\{i: i>1, B_i <_{p} B_{i+1}\}|}, \\
	D_{\ell f} (x,y,t)&:=&1 + \sum_{n \geq 1} t^n 
	\sum_{\pi \in \WOPln} x^{\ell(\pi)} y^{|\{i: 1<i<\ell(\pi)-1, B_i <_{p} B_{i+1}\}|}.
\end{eqnarray*}

We will always use $D$, $D_\ell$, $D_f$ and $D_{\ell f}$ to abbreviate $D(x,y,t)$, $D_\ell(x,y,t)$, $D_f(x,y,t)$ and $D_{\ell f}(x,y,t)$. As we are generally looking at the same cases as Section 4.5, we shall briefly describe the classification of cases and give the contribution of each case.

For any $\pi=B_1/\cdots/B_j\in\WOPln$, we let $w(\pi)=w_1\cdots w_n\in S_n(123)$. Let the first return of the corresponding Dyck path be at the \thn{n-k} column and let $B_i$ be the block containing the number $w_{n-k}$.

For the function $D(x,y,t)$, there are 4 cases.

\noindent{\bf Case 1.} Both $B_{i-1}$ and $B_i$ are of size $1$.\\
The contribution to  $D(x,y,t)$ is $t^2x^2yD^2.$

\noindent{\bf Case 2.} $w_{n-k-1}\notin B_{i}$ and $\pi$ does not satisfy Case 1.\\
The contribution to  $D(x,y,t)$ is 
$
txD\left(D+\frac{D_f -1}{x}\right)-t^2x^2D^2.
$

\noindent{\bf Case 3.} $w_{n-k-1}\in B_i$ and $w_{n-k+1}\notin B_i$.\\
The contribution to  $D(x,y,t)$ is 
$\left(D-1-xt\left(D+\frac{D_\ell -1}{x}\right)\right)\cdot tD.$

\noindent{\bf Case 4.} $w_{n-k-1}\in B_i$ and $w_{n-k+1}\in B_i$.\\
The contribution to  $D(x,y,t)$ is 
$\left(D_\ell-1-xt\left(D+\frac{D_\ell -1}{x}\right)\right)\cdot t\frac{D_f-1}{x}.$

Summing the contribution of all the 4 cases, we have
\begin{eqnarray}\label{t23eq1}
	D(x,y,t)&=&1+t^2x^2(y-1)D^2+txD\left(D+\frac{D_f -1}{x}\right)\nonumber\\
	&&+tD\left(D-1-xt\left(D+\frac{D_\ell -1}{x}\right)\right) \nonumber\\
	&&+t \left(\frac{D_f-1}{x}\right) \left(D_\ell-1-xt\left(D+\frac{D_\ell -1}{x}\right)\right).
\end{eqnarray}

For the function $D_\ell(x,y,t)$, there are 6 cases.

\noindent{\bf Case 1.} Both $B_{i-1}$ and $B_i$ are of size $1$, and $k>0$.\\
The contribution to  $D_\ell(x,y,t)$ is $t^2x^2yD(D_\ell-1).$

\noindent{\bf Case 2.} $w_{n-k-1}\notin B_{i}$, $k>0$, and $\pi$ does not satisfy Case 1.\\
The contribution to  $D_\ell(x,y,t)$ is 
$
txD\left(D_\ell-1+\frac{D_{\ell f} -1}{x}\right)-t^2x^2D(D_\ell-1).
$

\noindent{\bf Case 3.} $w_{n-k-1}\in B_i$, $k>0$, and $w_{n-k+1}\notin B_i$.\\
The contribution to  $D_\ell(x,y,t)$ is 
$\left(D-1-xt\left(D+\frac{D_\ell -1}{x}\right)\right)\cdot t(D_\ell-1).$

\noindent{\bf Case 4.} $w_{n-k-1}\in B_i$ and $w_{n-k+1}\in B_i$.\\
The contribution to  $D_\ell(x,y,t)$ is 
$\left(D_\ell-1-xt\left(D+\frac{D_\ell -1}{x}\right)\right)\cdot t\frac{D_{\ell f}-1}{x}.$

\noindent{\bf Case 5.} $k=0$ and $w_{n-k-1}\notin B_i$.\\
The contribution to  $D_\ell(x,y,t)$ is $txD$.

\noindent{\bf Case 6.} $k=0$ and $w_{n-k-1}\in B_i$.\\
The contribution to  $D_\ell(x,y,t)$ is  $t\left(D_\ell-1-tx\left(D+\frac{D_\ell-1}{x}\right)\right).$

Summing the contribution of all the 6 cases, we have
\begin{eqnarray}\label{t23eq2}
	D_\ell(x,y,t)&=&1+txD+t^2x^2(y-1)D(D_\ell-1)+txD\left(D_\ell-1+\frac{D_{\ell f} -1}{x}\right)\nonumber\\
	&&+t(D_\ell-1)\left(D-1-xt\left(D+\frac{D_\ell -1}{x}\right)\right) 
	 \nonumber\\
	 &&+t\left(\frac{D_{\ell f}-1}{x}\right)\left(D_\ell-1-xt\left(D+\frac{D_\ell -1}{x}\right)\right)\nonumber\\
	 && +t\left(D_\ell-1-tx\left(D+\frac{D_\ell-1}{x}\right)\right).
\end{eqnarray}

The functions $D_f(x,y,t)$ and $D_{\ell f}(x,y,t)$  have exactly the same 4 cases and 6 cases as $D(x,y,t)$ and $D_{\ell}(x,y,t)$. The main difference on the right hand side expansion is that some $D$ and $D_\ell$ become $D_f$ and $D_{\ell f}$. We omit the classification of cases and organize the terms of the expressions of $D_f(x,y,t)$ and $D_{\ell f}(x,y,t)$ in the same way as $D(x,y,t)$ and $D_{\ell}(x,y,t)$, and we have
\begin{eqnarray}\label{t23eq3}
	D_f(x,y,t)&=&1+t^2x^2(y-1)(D_f-1)D+txD_f\left(D+\frac{D_f -1}{x}\right)\nonumber\\
	&&+tD\left(D_f-1-xt\left(D_f+\frac{D_{\ell f} -1}{x}\right)\right) \nonumber\\
	&&+t \left(\frac{D_f-1}{x}\right) \left(D_{\ell f}-1-xt\left(D_f+\frac{D_{\ell f} -1}{x}\right)\right),
\end{eqnarray}
and
\begin{eqnarray}\label{t23eq4}
	D_{\ell f}(x,y,t)&=&1+txD_f+t^2x^2(y-1)(D_f-1)(D_\ell-1)+txD_f\left(D_\ell-1+\frac{D_{\ell f} -1}{x}\right)\nonumber\\
	&&+t(D_\ell-1)\left(D_f-1-xt\left(D_f+\frac{D_{\ell f} -1}{x}\right)\right) 
	\nonumber\\
	&&+t\left(\frac{D_{\ell f}-1}{x}\right)\left(D_{\ell f}-1-xt\left(D_f+\frac{D_{\ell f} -1}{x}\right)\right)\nonumber\\
	&& +t\left(D_{\ell f}-1-tx\left(D_f+\frac{D_{\ell f}-1}{x}\right)\right).
\end{eqnarray}

Using equations (\ref{t23eq1}), (\ref{t23eq2}), (\ref{t23eq3}) and (\ref{t23eq4}), one can compute the Groebner basis of the functions  to find an equation that $D(x,y,t)$ satisfies, and we have the following theorem.

\begin{theorem}The function $\mathbb{WOP}^{\pdes}_{321}(x,y,t)$ is the root of the following degree $6$ polynomial equation about $D$:
%
%\noindent
%$D((-1+D) x+t D(-1-D^2 (1+x)^2+2 D D(1+x+x^2 (-1+y)D)-2 x^2
%(-1+y)D)+D^3 t^5 x^5 (-1+y)^3+D^2 t^4 x^3 (-1+y)^2 (-2+2 D (1+x)+x
%(1+x-x y))+t^2 D(1+x+D (-2+x (2 (-2+y)+x (4+x (-1+y)) (-1+y)))-x
%D(x^2 (-1+y)^2+yD)-D^2 (1+x) (-1+x (-2+3 x (-1+y)+y))D)+D
%t^3 x (-1+y) D(1+D^2 (1+x)^2+2 x (-1+x (-1+y))+D D(-2+x^2 (4+3
%x-2 y-3 x y)D)D)D) D(1+D D(-2+D D(1+t
%D(1+x-t D(1+x+x^2D)+D (-1+t) D(1+x+t x^2
%(-1+y)D)+t x^2 yD)D)D)D)=0.
%$	
\begin{multline*}
\left((-1+D) x+t \left(-1-D^2 (1+x)^2+2 D \left(1+x+x^2 (-1+y)\right)-2 x^2
(-1+y)\right)+D^3 t^5 x^5 (-1+y)^3\right.\\
+D^2 t^4 x^3 (-1+y)^2 (-2+2 D (1+x)+x
(1+x-x y))+t^2 \left(1+x+D (-2+x (2 (-2+y)\right.\\
\left.+x (4+x (-1+y)) (-1+y)))-x
\left(x^2 (-1+y)^2+y\right)-D^2 (1+x) (-1+x (-2+3 x (-1+y)+y))\right)\\
\left.+D
t^3 x (-1+y) \left(1+D^2 (1+x)^2+2 x (-1+x (-1+y))+D \left(-2+x^2 (4+3
x-2 y-3 x y)\right)\right)\right)\\
\left(1+D \left(-2+D \left(1+t
\left(1+x-t \left(1+x+x^2\right)+D (-1+t) \left(1+x+t x^2
(-1+y)\right)+t x^2 y\right)\right)\right)\right)=0.
\end{multline*}
\end{theorem}
We can use Mathematica to compute the following generating function:
\begin{multline*}
\mathbb{WOP}^{\pdes}_{321}(x,y,t)=1+t x+t^2 \left(x^2 y+x^2+x\right)+t^3 \left(4 x^3 y+x^3+x^2 y+5
x^2+x\right)+t^4
\left(2 x^4 y^2+11 x^4 y\right.\\\left.+x^4+11 x^3 y+16 x^3+x^2
y+13 x^2+x\right)
+t^5 \left(15 x^5 y^2+26 x^5 y+x^5+5 x^4 y^2\right.\\\left.+65 x^4 y+42
x^4+23 x^3 y+76 x^3+x^2 y+29 x^2+x\right)+\cdots.
\end{multline*}

\section{Open problems}
In this paper, we mainly use the classical recursion of 132-avoiding permutations and the Dyck path bijection of 123-avoiding permutations to prove results on the generating functions of ordered set partitions that word-avoid some patterns of length $3$ tracking several statistics.  Our definition of word-avoidance of an ordered set partition differs from the pattern avoidance defined by Godbole, Goyt, Herdan and Pudwell \cite{GGHP}. Notwithstanding, our definition of 321-word-avoiding ordered set partition coincides $\alpha$-avoiding ordered set partition in the sense of \cite{GGHP} for any pattern $\alpha\in S_3$.

Due to this coincidence, we spent much of this paper
dealing with the set $\WOP_n(321)$ of ordered set partitions word-avoiding 321. In Section 3, we solved all the generating functions tracking the statistic {\em descent} about $\WOP_n(\alpha)$ for any pattern $\alpha$ of length 3, and obtained many beautiful symmetries and formulas with multinomial coefficients. However, the enumeration for $\wop_{[b_1, \ldots, b_k]}(321)=op_{[b_1, \ldots, b_k]}(321)$ and $\wop_{\langle b_1^{\alpha_1},\ldots ,b_k^{\alpha_k}\rangle}(321)=op_{\langle b_1^{\alpha_1},\ldots ,b_k^{\alpha_k}\rangle}(321)$ are still open. As a first question, an explicit formula for $\wop_{\langle b_1^{\alpha_1},\ldots ,b_k^{\alpha_k}\rangle}(321)$ is desired.

In Section 4 and Section 5, we got nice results for all the generating functions tracking the statistics \mindes\ and \pdes, except that we did not have any result about $\mathbb{WOP}^{\pdes}_{123}(x,y,t)$. In particular, we had polynomial equations about the generating functions $\mathbb{WOP}^{\mindes}_{321}(x,y,t)$ and $\mathbb{WOP}^{\pdes}_{321}(x,y,t)$ stated in Section 4.5 and Section 5.3, which would still make sense when using pattern avoidance definition in the sense of \cite{GGHP}.  The polynomial equations have all the information of the generating functions, and one can come up with efficient recursions easily with the equations. The open problem in this part is the function $\mathbb{WOP}^{\pdes}_{123}(x,y,t)$. We have not been able to get recursions about $\mathbb{WOP}^{\pdes}_{123}(x,y,t)$ since the \pdes\ statistic changes abnormally at the action {\em lift}.

\end{document}